\documentclass[11pt,oneside,reqno]{amsart}
\usepackage[latin1]{inputenc}
\usepackage[english]{babel}
\usepackage{mathrsfs}
\usepackage{indentfirst}
\usepackage{paralist}
\usepackage{amssymb}
\usepackage{amsthm}
\usepackage{amsmath}
\usepackage{amscd}
\usepackage{graphicx}
\usepackage[colorlinks=true]{hyperref}
\usepackage[margin=1in]{geometry}
\usepackage{cite}
\usepackage{lineno} 
\usepackage{enumitem} 
\usepackage{multicol}
\usepackage{xcolor}
\definecolor{ForestGreen}{rgb}{0.13, 0.55, 0.13}

\newtheorem{theorem}{Theorem}[section] 
\newtheorem*{theorem*}{Main Result}
\newtheorem{thm*}{Known result}
\newtheorem{corollary}[theorem]{Corollary} 

\newtheorem{proposition}[theorem]{Proposition}
\newtheorem{definition}[theorem]{Definition}

\theoremstyle{definition}

\theoremstyle{remark}
\newtheorem{remark}[theorem]{Remark}

\numberwithin{equation}{section} 
\def\z{{\bf z}}
\DeclareMathOperator{\tail}{Tail}
\newcommand{\R}{\ensuremath{\mathbb{R}}}
\newcommand{\Rn}{\ensuremath{\mathbb{R}^n}}
\newcommand{\N}{\ensuremath{\mathbb{N}}}

\newcommand{\Co}{\mathcal C}
\newcommand{\E}{\mathcal E}
\newcommand{\F}{\mathcal F}
\newcommand{\Fi}{\mathcal{F}_1^s}

\newcommand{\W}{\mathcal W}

\renewcommand{\P}{\mathcal P_s}
\newcommand{\J}{\mathcal J_s}
\DeclareMathOperator{\Per}{Per}
\DeclareMathOperator{\sgn}{sgn}

\DeclareMathOperator{\diam}{diam}

\renewcommand{\le}{\leqslant}
\renewcommand{\leq}{\leqslant}

\renewcommand{\geq}{\geqslant}

\title[Solutions of the fractional $1$-Laplacian]{Solutions of the fractional $1$-Laplacian: existence, asymptotics and flatness results}
\begin{document}
\author[C. Bucur]{Claudia Bucur}
\address[C.\ Bucur]{\newline\indent Dipartimento di Matematica Federigo Enriques
	\newline\indent
	Universit\`a degli Studi di Milano
	\newline\indent
	Italy, Milano, Via Saldini 50, 20133}
\email{\href{mailto:claudia.bucur@unimi.it}{claudia.bucur@unimi.it}} 

\thanks{The author is a member
	of {\em Gruppo Nazionale per l'Analisi Ma\-te\-ma\-ti\-ca, la Probabilit\`a e le loro Applicazioni} (GNAMPA) 
	of the {\em Istituto Nazionale di Alta Matematica} (INdAM). The author has been supported by INdAM-GNAMPA, project E53C23001670001.}

\begin{abstract}
In this paper, we study the existence of solutions for the equation $(-\Delta)_1^s u=f$ in a bounded open set with Lipschitz boundary $\Omega\subset \Rn$,  vanishing on $\Rn \setminus \Omega$, given some $s\in (0,1)$. Contextually, we obtain that the sequence of solutions for $(-\Delta)_p^s u=f$ convergences  to a solution of $(-\Delta)_1^s u=f$ when $p\to 1$.  We obtain our existence and convergence results by comparing the $L^{\frac{n}{s}}$ norm of $f$ to $(2S_{n,s})^{-1}$, where $S_{n,s}$ is the sharp fractional Sobolev constant, or, when $f$ is non-negative, a weighted version of the fractional Cheegar constant to $1$, and in this case, the results are sharp. We further prove that solutions are ``flat" on sets of positive Lebesgue measure. 
\end{abstract}
\maketitle

\section{Introduction}

In this note, we address some issues concerning minimizers and weak solutions of an equation related to the fractional $1$-Laplacian. We consider (unless explicitly otherwise stated) a bounded open set $\Omega \subset \Rn$ with Lipschitz boundary, a fixed fractional parameter $s\in (0,1)$ and a given function $f$ in $ L^{\frac{n}{s}}(\Omega)$. We deal with the nonlocal problem
\begin{equation}  
\left\{\begin{aligned}
\label{mainpb}\displaystyle (-\Delta)^s_1 u & =f &&\mbox{ in } \Omega\\
u&  =0 && \mbox{ in } \Co \Omega,\end{aligned}\right.\end{equation}
where $\Co \Omega$ is the usual notation for $\Rn \setminus \Omega$.
The fractional $1$-Laplacian arises in the Euler-Lagrange equation related to functions of least $W^{s,1}$-energy and could be thought, roughly speaking, as
\begin{equation}  \begin{aligned} \label{defn1} (-\Delta)^s_1 u(x)= \int_{\Rn} \frac{u(x)-u(y)}{|u(x)-u(y)|} \frac{dx dy}{|x-y|^{n+s}},
\end{aligned}\end{equation}
by simply taking $p=1$ in the definition of the fractional $p$-Laplacian. We recall that up to constants, the fractional $p$-Laplacian for some $p>1$ is defined as
\[
 (-\Delta)^s_p u(x):=P.V. \int_{\Rn} \frac{|u(x)-u(y)|^{p-2}(u(x)-u(y))}{|x-y|^{n+sp}} \, dx dy.\]

\smallskip

 Here,  we continue the ongoing research started by the author and collaborators in \cite{bdlv,bdlvm} 
 regarding $(s,1)$-harmonic functions with non-vanishing boundary data $\varphi \colon \Co \Omega \to \R$  
 and minimizers of the related energy  
\begin{equation}  \begin{aligned}\label{kenny} \E_1^s(u,\Omega)= \frac{1}{2}\iint_{Q(\Omega)} \frac{|u(x)-u(y)|}{|x-y|^{n+s}}\, dxdy  ,\end{aligned}\end{equation}
where $Q(\Omega):= \R^{2n} \setminus (\Co \Omega)^2$. 
Regarding minimizers of \eqref{kenny} -- functions of least $W^{s,1}$-energy -- the author and collaborators proved the fractional counterpart of some results   from \cite{BDG,SWZ,J}.  Precisely  we showed in \cite{bdlv} that the
level sets of the minimizers are nonlocal minimal surfaces and
that minimizers exist, assuming that the exterior given data has an ``integrable tail''. In \cite{bdlvm} we proved that the sequence of minimizers of the energy
\begin{equation}  \begin{aligned}\label{ken} \E_p^{s}(u,\Omega)= \frac{1}{2p}\iint_{Q(\Omega)} \frac{|u(x)-u(y)|^{p}}{|x-y|^{n+sp}}\, dxdy  \end{aligned}\end{equation}
  converges to a minimizer of \eqref{kenny} as the parameter $p$ tends to 1. 
\\Concerning weak solutions of \eqref{mainpb}, their definition was originally introduced in \cite{mazfr} for a $L^2$ right-hand side data. In \cite{bdlvm}, the existence of $(s,1)$-harmonic functions, together with their equivalence to minimizers is investigated. For more comprehensive details,  interested readers can also  refer to the survey  \cite{bws1},  where these findings are sketched  and the classical problem is also discussed.

\medskip

In this paper, also in view of these previous results, we wondered whether studying the asymptotics as $p\to 1$ of the $p$-problem could lead to the existence of  solutions of \eqref{mainpb}. To be specific, our fractional setting is the following: for any $p\in [1,+\infty)$ we take
\begin{equation*} 
\mathcal F_p^{s}(u):=  \frac{1}{2p}\iint_{\R^{2n}} \frac{|u(x)-u(y)|^{p}}{|x-y|^{n+sp}}\, dxdy -\int_\Omega fu \, dx.
 \end{equation*} 
For $p>1$, minimizers of $\F_p^s$  are known to be  weak solutions of 
\begin{equation} \label{mainpbp}
\left\{\begin{aligned} 
(-\Delta)^{s}_pu& =f &&\mbox{ in } \Omega\\
u&= 0 && \mbox{ in } \Co \Omega.
\end{aligned} \right.
\end{equation} 
 We also note that the minimizer  is unique as a consequence of the strict convexity of the functional.

\noindent The study of the existence of global minimizers for $\F_1^s$ seems  to extend beyond direct methods of the calculus of variations. Notably, the energy does not always present a bound from below, and existence seems to depend, as in the classical case $s=1$, on some characteristics of $f$ and $\Omega$. We approach the problem from two perspectives. For any $f\in {L^{\frac{n}{s}}(\Omega)}$, we compare  the norm of $f$ to $(2 S_{n,s})^{-1}$, where $ S_{n,s}$ is the sharp fractional Sobolev constant. We prove existence and asymptotics as $p\to 1$ of the associated $p$-problem when the norm is not larger than such a  constant, leaving open the case when the norm is larger. On the other hand, for $f \geq 0$, we are able to provide sharp existence and asymptotic results, by comparing to $1$ the weighted fractional Cheegar constant. 

To be more precise, for $f\in {L^{\frac{n}{s}}(\Omega)}$,
let $u_p$ denote the unique minimizer of $\F_p^{s_p}$ and weak solution of \eqref{mainpbp} that vanishes on $\Co \Omega$, for $p\in (1, c_{n,s})$ with $c_{n,s}>1$ such that
\begin{equation}  \begin{aligned} \label{sppar} s_p:= n+s-\frac{n}{p} \in (s,1)\end{aligned}\end{equation}
and $s_pp<1$.
 If
\[ \|f\|_{L^{\frac{n}{s}}(\Omega)}<(2 S_{n,s})^{-1}\qquad \mbox{ then } \quad   u_p \xrightarrow[p \to 1]{} u_1=0\]  
in the $L^1(\Omega)$ norm, where $u_1=0$ is the unique minimzer of $\F_1^s$ and weak solution of \eqref{mainpb}. If
\[ \|f\|_{L^{\frac{n}{s}}(\Omega)}=(2 S_{n,s})^{-1} \quad \mbox{ then } \quad    u_p \xrightarrow[p \to 1]{} u_1\] 
 in the $L^1(\Omega)$ norm, where $u_1$ is a minimizer of $\F_1^s$ and weak solution of \eqref{mainpb}. 
 We give examples that in this case, non-vanishing minimizers exist, and also that when
$\|f\|_{L^{\frac{n}{s}}(\Omega)}>(2 S_{n,s})^{-1}$
 the energy may be unbounded from below, and global minimizers may not exist. Notice furthermore that such a condition is actually necessary for lack of existence, check \eqref{fbk}.
These results are the content of Theorem \ref{exmin}.

\smallskip

In the case $f \geq 0$, we are able obtain sharp results, by relying on another interesting problem, that of Cheegar sets. To be more precise, we take a non-negative $f\in L^{\frac{n}{\sigma}}(\Omega)$  for some $\sigma \in (0,s)$ and recall the definition of the fractional perimeter of the set $E\subset  \Rn$,  introduced in \cite{nms}, as 
\begin{equation}  \begin{aligned} \label{fracp} \Per_s(E, \Omega) &  := \frac{1}{2}\iint_{Q(\Omega)} \frac{|\chi_E(x)-\chi_E(y)|}{|x-y|^{n+s}} dx \, dy, \end{aligned}\end{equation}
where $\chi_E$ is the characteristic function of the set $E$, and we denote \[ |E|_f= \int_E f \, dx.\] 
We point out that for $E\subset \Omega$, 
$\Per_s(E,\Omega)=\Per_s(E,\Rn)=:\Per_s(E)$.\\
We slightly modify the renown  problem of the $s$-Cheegar constant (see \cite{brch,cheeg}),  by adding the weight $f$ in the contribution to the volume. We call the $(s,f)$-Cheegar constant
\begin{equation}  \begin{aligned} \label{cheeg} h^f_s(\Omega):=\inf \left\{ \frac{\Per_s(A)}{|A|_f}\, \bigg| \,\,  A\subset \Omega, \, |A|_f>0\right\},\end{aligned}\end{equation}
and say that $\tilde E\subset \Omega$ is a $(s,f)$-Cheegar set if
\[ h^f_s(\Omega):=\frac{\Per_s(\tilde E)}{|\tilde E|_f}.\]
We obtain the following sharp existence and asymptotic result.
Let again $u_p$ denote the unique minimizer of $\F_p^{s_p}$. 
If
\[ h_s^f(\Omega)> 1 \qquad \mbox{ then } \qquad   u_p \xrightarrow[p \to 1]{} u_1=0\]  
in the $L^1(\Omega)$ norm, where $u_1=0$ is the unique minimizer of $\F_1^s$ and weak solution of \eqref{mainpb}. If
\[ h_s^f(\Omega)= 1,\qquad \mbox{ then } \qquad  u_p \xrightarrow[p \to 1]{} u_1\]  
in the $L^1(\Omega)$ norm, where $u_1$ is a minimizer of $\F_1^s$ and weak solution of \eqref{mainpb}. In this case, we also provide an example showing non-uniqueness of solutions. 
 If
\[ h_s^f(\Omega)< 1\qquad \mbox{ then } \qquad  [u_p]_{W^{s_p,p}(\Omega) } \xrightarrow[p\to 1]{} +\infty,\] 
and minimizers of $\F_1^s$ do not exist. We insert these findings in Theorem \ref{ass}. 

\medskip

 We compare our findings to those in the classical framewok. The  asymptotic results of the type we are interested in were first investigated in \cite{kawohl}, where the author studied the behavior, as $p\to 1$, of weak solution $u_p$ of the torsion problem for the $p$-Laplacian $-\Delta_p u=1$ in $\Omega$.  The author observed that, if  $\Omega$ is sufficiently small, then  \begin{equation}  \begin{aligned}\label{up0} u_p \xrightarrow[p \to 1]{} 0
\end{aligned}\end{equation}
uniformly in $\Omega$, 
while for $\Omega$ large enough
\begin{equation}  \begin{aligned}\label{infty} u_p \xrightarrow[p \to 1]{} +\infty.\end{aligned}\end{equation}
 The research is continued in \cite{trombetti}, where the authors considered a right hand side $f\in L^n(\Omega)$
and prove that when $\|f\|_{L^n(\Omega)} \leq 1/S_{n}$
 with $S_n$ being the sharp Sobolev constant, then \eqref{up0} holds for weak solutions.  Additionally, minimizers of the $p$-energy approach minimizers of the $1$-energy, as $p\to 1$, even though the $1$-minimizer might not vanish. 
We point out that these classical results  rely on an explicit formula for solutions of the $p$-Laplace equation on the ball, which is not available in the fractional case. We also remark that the constant $2S_{n,s}$ in our result is as sharp as in the classical case, given that the multiplicative term $1/2$ depends solely on the use of this constant in \eqref{ken} introduced for symmetry purposes.

 Regarding studying asymptotics by comparing the Cheegar constant $h(\Omega)$ to 1, this was done in the classical case in \cite{bueno,bueno2}, only when $f$ has the constant value $1$. In these two papers, the authors emphasized the strong connection between the Cheegar constant and $p$-torsion functions, i.e. unique solutions of $(-\Delta)_p u=1$ in $\Omega$, showing that
 \begin{equation}  \begin{aligned} \label{kmnj}\lim_{p\to 1}\|\phi_p\|_{L^1(\Omega)}^{1-p}=\lim_{p\to 1} \|\phi_p\|_{L^\infty(\Omega)}^{1-p}=h(\Omega)
 ,\end{aligned}\end{equation}
 and also that 
 $\phi_p $, renormalized by its $L^1$ norm, converges to a solution of the $(-\Delta)_1 u=1$ as $p\to 1$. From this, comparing $h(\Omega)$ to $1$, $\phi_p$ are proved to converge to either zero or $+\infty$, in the $L^1, L^\infty$ norm. 
 
We remark that with respect to the classical case, the nonlocal nature of our problem significantly complicates the situation. In particular, one would need a uniform bound in $p$ of the $W^{s,1}$-seminorm of the minimizer $u_p$. In the classical case, this follows directly from  H\"older and Sobolev inequalities. In the nonlocal case, the bound for the local contribution $\Omega\times \Omega$ is straight forward, while the difficulty arises when dealing with the nonlocal contribution on the unbounded set $\Omega \times \Co \Omega$. A careful inspection of the nonlocal contribution is therefore a first crucial technical step in order to achieve our result. as carried out in Proposition \ref{ub}.

 \smallskip
 
To best of our knowledge, the results presented in this paper are completely new. We shortly describe what was known in the fractional case. The fractional Cheegar constant was introduced in \cite{brch}, where the connection between a variational formulation for the torsion problem for the $(s,1)$-Laplacian and the eigenvalue problem $(-\Delta)_p^s u = \lambda |u|^{p-2}u$ is studied. In \cite{cheeg}, an alternative characterization of the fractional Cheegar constant is provided, by studying the $(s,p)$-torsional problem $(-\Delta)_p^s u = 1$ and obtaining \eqref{kmnj} in the fractional case.
The difference with our results -- besides our use of a general term $f$ -- is that we prove that the sequence of $u_p$, minimizers of $\F_p^{s_p}$, converges to minimizer and a weak solution of $\F_1^s$ when $h_s(\Omega) \geq 1$. 
Moreover, the approach we use is different from that of \cite{cheeg}, which is based on results from \cite{brch}.  Our approach, more similar to \cite{bueno2}, is self-contained, and fully uses the fractional parameter $s_p$, which plays a significant role in our investigation. 
We point our furthermore that neither of \cite{brch,cheeg} are interested in solutions of the problem \eqref{mainpb}, nor mention the existence of weak solutions.  We point out also the very recent paper \cite{another}, that appeared shortly after the present work, where the authors study renormalized solutions of the fractional $1$-Laplacian.

\bigskip

 To give some further input on weak solutions, we also discuss a so called ``flatness" results. 
The difficulty of defining a weak solution of the fractional $1$-Laplacian is evident looking at \eqref{defn1} -- the quotient $(u(x)-u(y))/|u(x)-u(y)|$, i.e. the sign function of $(u(x)-u(y))$, does not have a meaning  when $u(x)=u(y)$. The definition of a weak solution cannot be given by using purely an  integro-differential equation, rather it is necessary, as done in \cite{mazfr}, to require that there exists a multivalued function $\mathbf z\colon \R^{2n} \to [-1,1]$ 
equal to 
$\mbox{sgn}(u(x)-u(y))$, where $\mbox{sgn}$ is the generalized sign function. A  concern of this paper is to prove that generally, such a definition cannot be simplified, since solutions are ``flat" on sets of positive Lebesgue measure.  Precisely, if $u$ is a weak solution of \eqref{defn1} 
then
the set $\{ (x,y) \in \R^{2n}\setminus (\Co \Omega)^2 \, | \, u(x)=u(y)\}$ has positive Lebesgue measure -- see Theorem \ref{thmone}. We prove a similar result for minimizers, showing that  $\{x\in  \Omega \, | \, |u(x)|=\|u\|_{L^\infty(\Omega)}\}$ has positive Lebesgue measure, in Theorem \ref{thmtwo}. Such results are the fractional counterpart of \cite{Ponce}, where the authors prove that a weak solution of the $1$-Laplacian equation with $L^n(\Omega)$ right hand side  and zero boundary data has a vanishing gradient on a set of positive Lebesgue measure.

\medskip

 We draw the reader's attention to the use of the particular  fractional parameter $s_p$ in \eqref{sppar},
first introduced in \cite{mazfr}. Such a choice appears necessary for technical reasons -- see Remark \ref{tech} -- but is enforced  by the fractional embedding
$ W^{s,1} (\Omega) \subset W^{s_p,p}(\Omega)$ (compare our result in Proposition \ref{ub} to \cite[Lemma 2.6]{brch} or \cite[Lemma 3.1]{bdlv}), in other words when sending  $p\to 1$ in the $(s,p)$-fractional problem, in the limit the fractional parameter $s$ has to ``decrease'' as well, and this ratio is described precisely by the choice of $s_p$. 
To further justify the use of $s_p$, in the Appendix we obtain that  for all $u\in \W^{s_q,q}_0(\Omega)$ for some fixed $q>1$ the pointwise limit holds
\begin{equation} \label{lim} \lim_{p\to 1}\E_p^{s_p}(u)= \E_1^s(u).\end{equation}

\medskip 

A very interesting (in our opinion) issue is also discussed in Section \ref{minper}: that of non-negative minimizers of $\F_1^s$  and sets $E \subset \Omega$ that minimize 
\[ \P(E)= \Per_s(E)- |E|_f.\] Whether it is true that $E$ is a minimal set for $\P$ if and only if $\chi_E$ is a minimizer of $\Fi$  and weak solution of \eqref{mainpb}, i.e. 
\[(-\Delta)^s_1 \chi_E =1,\]
is investigated,
together with the equivalence that
  $u$ is a non-negative minimizer of $\F_1^s$ if and only if any super-level set of $u$ is a minimal set for $\P$.  
  
\medskip

We highlight some open problems that are arguments for our future research.  It is not known if comparing the norm of $f$ to $(2S_{n,s})^{-1}$ yields an optimal result, in other words, it would be interesting  to see if there are functions $f$ or sets $\Omega$ such that 
$h^f_s(\Omega)\geq 1$, hence minimizers exist,  while the $L^{\frac{n}{s}}(\Omega)$ norm of $f$ is strictly greater than $(2S_{n,s})^{-1}$ (or $\Omega$ is large enough, corresponding to $f=1$). Exploring thus explicit examples of  functions $f$ and sets $\Omega$ for which $h_s^f(\Omega)$ can be computed is another key goal. In the classical case, a sharp asymptotic result is provided by comparing the norm of $f$ in the dual space  $W^{-1,\infty}(\Omega)$ to $1$, we have still to investigate the validity of a fractional counterpart.  The existence of weak solutions when 
$h_s^f(\Omega)<1$  remains an open question that we are committed  to investigating. 

We point out furthermore that, like in the classical case, the methods used in this paper are well suited for a null exterior data. We are not able to include a more general case of a given exterior data $\varphi \colon \Co \Omega \to \R$, which remains as an open problem.

\medskip 
 
In the rest of the paper, we proceed as follows. We use Section \ref{one} for the setting and some very useful tools. Section \ref{four} is dedicated to studying the limit case as $p\to 1$ when the $L^{\frac{n}s} $ norm of $f$ is sufficiently small, and prove the existence of a minimizer and a weak solution. We also deal with  the equivalence between minimizer and  weak solution, even in the absence of a bound on $\|f\|_{L^{\frac{n}s}(\Omega)}$. Section \ref{minper} contains the sharp results on existence and asymptotics by comparing the fractional Cheegar constant to $1$. We also 
discuss the relation to sets that minimize $\P$.  Moreover, we give examples of existence of non-trivial solutions and of non-existence when the $L^{\frac{n}s} $ norm of $f$ is sufficiently large, coinciding with the case $h_s(f) \leq 1$.  In Section \ref{two} we discuss the ``flatness" of weak solutions  and of minimizers. The Appendix  contains the proof of \eqref{lim} and some basic knowledge on $(s,p)$-minimizer/weak solution.

\section{Setting of the problem, tools and remarks}\label{one}

Let $n\geq 1$, $0<s<1\leq p<+\infty$ and let $\Omega\subset \Rn$ be a bounded open set with Lipschitz boundary (unless otherwise mentioned). Since we mainly look at asymptotics as $p\to 1$, we will emphasize some useful properties when $p$ is close enough to one, precisely when $sp<1$. 

We will use the notations 
\[ \omega_n= \mathcal H^{n-1}(\partial B_1), \qquad B_r=\{ x\in \Rn \, : \, |x| <r\} \; \mbox{ for some  } r>0 \]
and recall that $\mathcal L^n(B_1):=|B_1|= \omega_n/n$. We also denote
\[Q(\Omega):= \R^{2n} \setminus (\Co \Omega)^2.\]
We use the notation for the fractional Sobolev space $\W^{s,p}(\Omega)$,  
\[ \W^{s,p}(\Omega):=\left\{ u\colon \Rn \to \R  \, \bigg| \,u|_\Omega  \in L^p(\Omega), \int_{\Omega}\int_{\Omega}\frac{|u(x)-u(y)|^{p}}{|x-y|^{n+sp}}\, dxdy <+\infty\right\},\]
and we denote by
\begin{equation*} 
\begin{aligned} 
 &\; [u]_{W^{s,p}(\Omega)}=  \left(\int_{\Omega}\int_{\Omega}\frac{|u(x)-u(y)|^{p}}{|x-y|^{n+sp}}\, dxdy\right)^\frac{1}{p},
\\ 
&\; \|u\|_{W^{s,p}(\Omega)}= \left(\|u\|_{L^p(\Omega)}^p + [u]^p_{W^{s,p}(\Omega)}\right)^{\frac{1}{p}}
\end{aligned} 
\end{equation*}
the fractional (Gagliardo) $(s,p)$-seminorm, respectively norm. 
\\Of course, $\W^{s,p}(\Rn)= W^{s,p}(\Rn)$. 
\\We denote by $\W^{s,p}_0(\Omega)$  the closure of $C^\infty_c(\Omega)$ with respect to the $W^{s,p}(\Omega)$ norm. We point out that, when $sp<1$, it also holds
\begin{equation}  \label{space} \W^{s,p}_0(\Omega)=\left\{ u\colon \Rn \to \R  \, \bigg| \,u \in \W^{s,p}(\Omega), \,u=0 \mbox{ on } \Co \Omega\right\},
\end{equation}
check \cite[Proposition A.1]{teolu}.\\
What is more, when $sp<1$ and $u=0$ on $\Co \Omega$, the norms $W^{s,p}(\Omega)$ and $W^{s,p}(\Rn)$  are equivalent, i.e.   there exist $C_1,C_2>0$ (depending on $n,s,p,\Omega$) such that 
\begin{equation}  \begin{aligned}\label{equinorm}
 C _1\|u\|_{W^{s,p}(\Omega)}\leq [ u]_{W^{s,p}(\Rn)}   \leq C_2\|u\|_{W^{s,p}(\Omega)}
.\end{aligned}\end{equation}
To prove this, we recall a fractional Hardy inequality \cite[Theorem 1.4.4.3]{grisvard} or the more recent \cite[Proposition A.2]{teolu} and point out the very useful \cite[Corollary A.3]{teolu}.  We remark furthermore that the proof of the Hardy inequality in \cite{grisvard, teolu} requires $\Omega$ to have Lipschitz boundary, and we are unaware of proofs on sets with less regularity. 

\begin{theorem}\label{hard}
Let $1\leq p<1/s$. Then there exists $C=C(n,s,p, \Omega)$ such that 
\[ \int_{\Omega} \frac{ |u(x)|^p}{\mbox{dist}(x, \partial \Omega)^{sp}} \, dx \leq C \| u\|_{W^{s,p}(\Omega)}\]
for all $u\in W^{s,p}(\Omega)$.
\end{theorem}
\begin{proposition} \label{thefig}
Let $1\leq p<1/s$.  Then for every $u\in W^{s,p}(\Omega)$
\begin{equation}  \begin{aligned} \label{frach} \int_{\Omega} \left(\int_{\Co \Omega} \frac{|u(x)|^p}{|x-y|^{n+sp}}\, dy\right) dx \leq  C \|u\|^p_{W^{s,p}(\Omega)},\end{aligned}\end{equation}
where
$C=C(n,s,p,\Omega)>0$.
\end{proposition}

Due to \eqref{frach},
\begin{equation*} 
\begin{aligned} 
[ u]^p_{W^{s,p}(\Rn)} = [u]^p_{W^{s,p}(\Omega)} + 2 \int_\Omega \left(\int_{\Co \Omega} \frac{|u(x)|^p}{|x-y|^{n+sp}}\, dy\right) dx \leq C_2\|u\|^p_{W^{s,p}(\Omega)},
\end{aligned}
\end{equation*}
while
\begin{equation} \label{glow} \begin{aligned}[ u]^p_{W^{s,p}(\Rn)} \geq&\;  [u]^p_{W^{s,p}(\Omega)}  +2  \int_\Omega  |u(x)|^p \left(\int_{\Co B_{\diam(\Omega)}(x) } \frac{dy }{|x-y|^{n+sp}}\right) dx 
\\
=&\;  [u]^p_{W^{s,p}(\Omega)}   + \|u\|^p_{L^p(\Omega)} 
\frac{\omega_n}{\diam(\Omega)^{sp}sp} \geq C_1\|u\|^p_{W^{s,p}(\Omega)},
\end{aligned}
\end{equation}
since  for all $x\in \Omega$, it holds that $\Co B_{\diam(\Omega)}(x)   \subset \Co \Omega$.
Thanks to this, we have that for $sp<1$,
\begin{equation}  \begin{aligned} \label{yusp} \W_0^{s,p}(\Omega)= \overline{C_0^\infty(\Omega)}^{[\cdot]_{W^{s,p}(\Rn)}}.\end{aligned}\end{equation}
Notice furthermore that 
for all $R>0$ such that $\Omega \subset \subset B_R$ we have that
\begin{equation}  \begin{aligned} \label{glows1} \|u\|_{W^{s,p}(B_R)} \leq C [u]_{W^{s,p}(\Rn)}.\end{aligned}\end{equation}
\medskip

We recall now the Sobolev inequality with sharp constants, see \cite[Corollary 4.2]{Frank}. We denote by $\dot{W}^{s,p}(\Rn)$ the closure of $C^\infty_0(\Rn)$ in the $W^{s,p}(\Rn)$ seminorm.
\begin{theorem}\label{sobb}
Let $n\geq 1$, $0<s<1 \leq p<n/s$. For all functions $u \in \dot{W}^{s,p}(\Rn)$ it holds that
\[ \|u\|^p_{L^{p^*}(\Rn)} \leq  S_{n,s,p} [u]^p_{W^{s,p}(\Rn)},\]
where
\begin{equation}  \begin{aligned} \label{sobct}
S_{n,s,p} = \frac{p}{p^*} \bigg(\frac{n}{\omega_n}\bigg)^{\frac{sp}{n}} \frac{1}{C_{n,s,p}}
\end{aligned}\end{equation}
with
\begin{equation*}
 \begin{aligned}
 C_{n,s,p}= 2\int_0^1 r^{ps-1} |1-r|^{\frac{n-ps}{p}}  \phi_{n,s,p}(r) \, dr,
\end{aligned}
\end{equation*}
and
\begin{equation*} 
\phi_{n,s,p}(r)= \left\{ 
\begin{aligned}  
&\;  |\mathbb S^{n-2}| \int_{-1}^1 \frac{(1-t^2)^{\frac{N-3}{2}}}{(1-2rt +r^2)^\frac{n+ps}{2}} \, dt, && N\geq 2 \\
&\;  \left( (1-r)^{-1-ps} - (1+r)^{-1-ps}  \right), && N=1 
\end{aligned}\right.
\end{equation*}
\end{theorem}
Notice that the Sobolev inequality holds in $\W_0^{s,p}(\Omega)$ for all $\Omega$ bounded open set. Indeed, if $u\in \W_0^{s,p}(\Omega)$ and has infinite $W^{s,p}(\Rn)$ seminorm we are done. Otherwise, by \eqref{glow} we have that $u\in W^{s,p}(\Rn)$. Given the equivalence of norms \eqref{equinorm}, then $u\in \dot{W}^{s,p}(\Rn)$. 


As a consequence of Theorem \ref{sobb} and recalling \eqref{fracp}, we get the fractional sharp isoperimetric inequality for sets $E\subset \Rn$, see  \cite{Frank,morini}.
\begin{theorem}
\label{iso}
Let $E \subset \Rn$ be a Borel set with finite Lebesgue measure. Then
\[ |E|^{\frac{n-s}{n}} \leq 2S_{n,s}\Per_s(E).\]
\end{theorem}
Notice also that
\begin{equation}  \begin{aligned} 
\label{isoeq} S_{n,s}= \frac{|B_1|^{\frac{n-s}{n}}} {2\Per_s(B_1)}= \frac{n-s}{n} \bigg(\frac{n}{\omega_n}\bigg)^{\frac{s}{n}} \frac{1}{C_{n,s,1}}.
\end{aligned}\end{equation}

\medskip

\noindent We define now our notions of minimizers and weak solutions. 

\begin{definition}\label{mnn} Let $0<s<1\leq p<+\infty$. We say that $u\in \W_0^{s,p}(\Omega)$ is an $(s,p)$-minimizer
if 
\[ \F_p^s(u) \leq \F_p^s(v)  \]
for all $v\in \W_0^{s,p}(\Omega)$.
\end{definition}

\noindent  Furthermore, we recall the following.

\begin{definition}
Let $0<s<1<p<+\infty$.  We say that $u\colon \Rn \to \R$ is a weak solution of the $(s,p)$-problem \eqref{mainpbp} if $u\in  \W_0^{s,p}(\Omega)$, and for all $w\in \W_0^{s,p}(\Omega)$
\[ \frac{1}2 \iint_{\R^{2n}}\frac{|u(x)-u(y)|^{p-2} (u(x)-u(y))(w(x)-w(y))}{|x-y|^{n+sp}} \, dx dy  = \int_\Omega fw\, dx.\] 
\end{definition}

\medskip

\noindent For the fractional $1$-Laplacian we give the following definition of weak solution -- check also \cite{mazfr,bdlvm}.

\begin{definition} \label{weaksol} Let $0<s<1$. 
 We say that a measurable function $u:\R^n\to\R$ is a weak solution of the problem \eqref{mainpb} 
if $u\in \W^{s,1}_0(\Omega)$ and 
if there exists 
\[ \z \in L^\infty(\R^{2n}) \qquad \Vert \z \Vert_{L^\infty(\R^{2n})} \leq 1, \qquad \z(x,y)=-\z(y,x),\] 
satisfying
\begin{equation}  \begin{aligned}\label{ee1}\frac12\iint_{\R^{2n}} \frac{\z(x,y)}{ \vert x - y \vert^{n+s}} (w(x) - w(y)) dx\, dy =\int_{\Omega} f(x) w(x) \, dx \\
 \qquad \qquad \hbox{for all }  w \in \W^{s,1}_0(\Omega),\end{aligned}\end{equation}
and
\begin{equation}\label{ee2}\z(x,y) \in \sgn(u(x) - u(y)) \quad \hbox{for almost all} \ (x, y) \in Q(\Omega).
\end{equation}
\end{definition}

We remark that since $u=0$ in $\Co \Omega$, the contribution to the  integral on $(\Co \Omega)^2$ is null, hence it is enough to require \eqref{ee2} to hold on $Q(\Omega)$.
We further point out  that as a consequence of Proposition \ref{thefig},  Definition \ref{weaksol} is well posed since 
	\begin{equation*} \begin{aligned} &\; \left| \iint_{\R^{2n}} \frac{ \z(x,y)}{|x-y|^{n+s}} (w(x)-w(y)) dx\, dy \right|
	\leq
	\iint_{\R^{2n}} \frac{ |w(x)-w(y)|}{|x-y|^{n+s}} dx\, dy
	\\ 
	=&\; 
	[w]_{W^{s,1}(\Omega)}+ 2  \int_{\Omega} \left(\int_{\Co \Omega}  \frac{ |w(x)|}{|x-y|^{n+s}} dy\right) dx \leq C \|w\|_{W^{s,1}(\Omega)}<+\infty,
\end{aligned}\end{equation*}
	for all $w\in W^{s,1}(\Omega)$. 

To conclude the introduction, we discuss the space defined in \eqref{space} and the necessity of having a domain with Lipschitz boundary. Notice that in the definition of $\W^{s,p}_0(\Omega)$ we only require, besides $u=0$ in $\Co \Omega$, only that the (local) norm $\| \cdot \|_{W^{s,1}(\Omega)}$ be finite, and deduce from this that the $W^{s,1}(\Rn)$ seminorm is finite as well (using the fractional Hardy inequality in Theorem \ref{hard} and the immediate Proposition \ref{thefig}). However, the fractional Hardy inequality  requires that $\Omega$ has Lipschitz boundary, as previously noted. Also, the choice of  $\W^{s,p}_0(\Omega)$ together with this regularity requirement on $\Omega$ ensures density of smooth functions \cite[Proposition A.1]{teolu}, as well as compact embedding of $W^{s,p}(\Omega)$ into $L^p(\Omega)$, both of which are essential in our proofs. \\ However, other choices are possible. Define for $0<p<1/s$ and for all open bounded sets $\Omega$ \begin{equation}
\label{newspace}
X_0^{s,p}(\Omega):=\{ u\in L^p(\Omega) \,\, |\,  \, [u]_{W^{s,p}(\Rn)}<+\infty, \, u=0 \mbox{ in } \Co \Omega\}.
\end{equation} This is Definition 2.2 in \cite{brch} and the definition (1.5) in \cite{fisc}. Notice that \[ X_0^{s,p}(\Omega) \subset \W^{s,p}_0(\Omega) ,\] see \eqref{glow}. Density of smooth functions in $X_0^{s,p}(\Omega)$ holds if $\Omega$ has a continuous boundary, in the precise sense given by \cite[Definition 4, Theorem 6]{fisc}. The compactness follows from \eqref{glows1} and the compact embedding of $W^{s,p}(B_R)$ into $L^p(B_R)$.   We point out that all our results work in such an alternative setting, thus in the space $X_0^{s,p}(\Omega)$ and $\Omega$ with continuous boundary.\\
 Another choice could be that of Definition 2.1 in \cite{brch}\[ \tilde W_0^{s,p}(\Omega):= \overline{C_0^\infty(\Omega)}^{[\cdot]_{W^{s,p}(\Rn)}},\]for all $\Omega \subset \Rn$ an open bounded set. On the one hand, this space is less general, i.e. \[\tilde W_0^{s,1}(\Omega) \subset X_0^{s,p}(\Omega) \subset \W_0^{s,p}(\Omega),\] on the other hand, the compactness result in \cite[Theorem 2.7]{brch} has the inconvenience that when $p=1$, one ends up in the space $X_0^{s,1}(\Omega)$ (which coincides with $\tilde W_0^{s,p}(\Omega)$ when $\Omega$ has Lipschitz boundary, see \cite[Lemma 2.3]{brch}). One would then need to adjust the proofs to work in these two spaces.\\
We kindly thank the referee for having raised the question whether it is necessary to ask $\Omega$ to be Lipschitz.

\section{Existence of minimizers and of weak solutions when $f\in L^{\frac{n}{s}}(\Omega)$}
 \label{four}

 In this section we prove the existence of minimizers and weak solutions (and their equivalence) when the $L^{\frac{n}s}(\Omega)$ norm of $f$ is sufficiently small. Contextually, we prove that  as $p\to 1$, the sequence of $(s_p,p)$-minimizers converges to an $(s,1)$-minimizer.

 We recall the definition of $s_p$ in \eqref{sppar} and prove first the following continuous embedding. 
 \begin{proposition}\label{ub} Let $\Omega \subset \Rn$ be a bounded open set. 
Let $u\in \W_0^{s_p,p}(\Rn)$. It holds that
 \[ \, [u]_{W^{s,1}(\Rn)} \leq C_{n,s,\Omega}^{\frac{1-p}{p}}[u]_{W^{s_p,p}(\Rn)},\]
 with $C_{n,s,\Omega}>0$.
 \end{proposition} 
 \begin{proof}
By employing H\"older, we have that
\begin{equation}  \begin{aligned} \label{obs1} \int_{\Omega} \int_{ \Omega} \frac{|u(x)-u(y)|^p}{|x-y|^{n+s_p p}} \, dx dy\geq [u]^p_{W^{s,1}(\Omega)} | \Omega|^{2(1-p)}.\end{aligned}\end{equation}

\noindent Denote for all $x\in \Omega$, $B_\Omega(x):= B_{\diam(\Omega)}(x)$ and notice that $\Co B_\Omega(x) \subset \Co \Omega$. We have that
\begin{equation*} \begin{aligned} &\; \int_\Omega \left(\int_{\Co \Omega} \frac{|u(x)|^p}{|x-y|^{n+s_p p}}\, dy\right)  dx  
= \int_\Omega \left( \int_{B_\Omega(x)\setminus \Omega}  \frac{|u(x)|^p}{|x-y|^{n+s_p p}} \, dy\right)  dx \\
&\; +  \int_\Omega \left(\int_{\Co B_\Omega(x)}  \frac{|u(x)|^p}{|x-y|^{n+s_p p}} \, dy\right)  dx 
.
\end{aligned}\end{equation*}
Using two times H\"older, for a fixed $x\in \Omega$,
\begin{equation*} \begin{aligned}
 \int_{B_\Omega(x)\setminus \Omega} \frac{|u(x)|}{|x-y|^{n+s}} \,dy  \leq \left(\int_{B_\Omega(x)\setminus \Omega} \frac{|u(x)|^p}{|x-y|^{(n+s)p}} \,dy \right)^{\frac{1}{p}} |B_\Omega(x)\setminus \Omega|^{\frac{p-1}{p}},
\end{aligned}\end{equation*}
and
\begin{equation*} \begin{aligned}
\int_{\Omega} \left(\int_{B_\Omega(x)\setminus \Omega} \frac{|u(x)|}{|x-y|^{n+s}} \,dy\right) dx \leq &\; 
\int_{\Omega} \left(\int_{B_\Omega(x)\setminus \Omega} \frac{|u(x)|^p}{|x-y|^{(n+s)p}} \,dy \right)^{\frac{1}{p}} |B_\Omega(x)\setminus \Omega|^{\frac{p-1}{p}} \, dx 
\\
\leq &\;\left[ \int_{\Omega} \left(\int_{B_\Omega(x)\setminus \Omega} \frac{|u(x)|^p}{|x-y|^{(n+s)p}} \,dy\right)  dx \right]^{\frac{1}{p}} \left(\int_\Omega |B_\Omega(x)\setminus \Omega| \,dx\right)^{\frac{p-1}{p}}
\\ =&\; \left[ \int_{\Omega} \left(\int_{B_\Omega(x)\setminus \Omega} \frac{|u(x)|^p}{|x-y|^{n+s_pp}} \,dy\right)  dx \right]^{\frac{1}{p}} M_\Omega^{\frac{p-1}{p}},
\end{aligned}\end{equation*}
recalling $(n+s)p=n+s_pp$ and denoting
\[ M_\Omega= \int_\Omega |B_\Omega(x) \setminus \Omega| \, dx. \]
 Integrating and again by H\"older, 
\begin{equation*}
 \begin{aligned} \int_{\Omega}\left(\int_{\Co B_\Omega(x)}  \frac{|u(x)|^p}{|x-y|^{n+s_p p}} \, dy\right) dx= & \, \|u\|_{L^p(\Omega)}^p \frac{\omega_n}{s_pp(\diam(\Omega))^{s_pp}} 
\\
\geq &\;\|u\|_{L^1(\Omega)}^p |\Omega|^{1-p}  \frac{\omega_n}{s_pp(\diam(\Omega))^{s_pp}}
\\
=:&\; |\Omega|^{1-p} \gamma_{n,s,p}^p   \left[\int_{\Omega} |u(x)| \left(\int_{\Co B_\Omega(x)} \frac{dy}{|x-y|^{n+s}} \right) dx\right]^p\end{aligned}\end{equation*}
where
 \begin{equation}  \begin{aligned} \label{gamma} \gamma_{n,s,p} = (\omega_n \diam(\Omega)^n)^{\frac{1-p}{p} }\frac{s}{(s_pp)^{\frac{1}{p}}}.\end{aligned}\end{equation} 
 We use the notations
\begin{equation}  \begin{aligned} \label{not} & A=[u]_{W^{s,1}(\Omega)}, \quad B= \int_\Omega \left(\int_{B_\Omega(x)\setminus \Omega} \frac{|u(x)|}{|x-y|^{n+s}} \,dy\right)  dx,  
\\ &\; C  =\|u\|_{L^1(\Omega)} \frac{\omega_n}{s\diam(\Omega)^s}=\int_{\Omega} |u(x)| \left(\int_{\Co B_\Omega(x)} \frac{dy}{|x-y|^{n+s}}\right) dx,\end{aligned}\end{equation}
and summing up the inequalities above, 
we have that
\begin{equation*} \begin{aligned} \frac{1}{2}[u]^p_{W^{s_p,p}(\Rn)}  \geq &\;  \frac{1}2 A^p |\Omega|^{2(1-p)} +  B^p M_\Omega^{1-p} + C^p |\Omega|^{1-p} \gamma_{n,s,p}^p
\\
\geq &\;  \min\left\{ |\Omega|^{2(1-p)}, |\Omega|^{1-p}, M_\Omega^{1-p}\right\} \left( \frac{1}2A^p + B^p + C^p  \gamma_{n,s,p}^p \right)
\\ \geq &\;  {4^{1-p}}\min\left\{ |\Omega|^{2(1-p)}, |\Omega|^{1-p}, M_\Omega^{1-p}\right\}  \left(2^{-\frac{1}{p}} A + B + C \gamma_{n,s,p}
\right)^p
		\\ \geq &\;
    {4^{1-p}}\min\left\{ |\Omega|^{2(1-p)}, |\Omega|^{1-p}, M_\Omega^{1-p}\right\} 2^{-p}\left( 2^{-\frac{1}p+1}A +2 B +2 C \gamma_{n,s,p}  \right)^p 
    \\ \geq &\;   {4^{1-p}}\min\left\{ |\Omega|^{2(1-p)}, |\Omega|^{1-p}, M_\Omega^{1-p}\right\} 2^{-p}\left( A +2 B +2 C \gamma_{n,s,p}  \right)^p 
  \end{aligned}\end{equation*}
  recalling  that $(a+b+c)^p \leq 4^{p-1}(a^p +b^p +c^p)$ for all $p>1, a, b, c \geq 0$ and  since $2^{-1/p+1}> 1$. 
Noticing that 
\[ \, [u]_{W^{s,1}(\Rn)} = A+2B +2C,\]
if
  \[ \gamma_{n,s,p} \geq 1,\] then
  \[ \, [u]^p_{W^{s_p,p}(\Rn)}\geq   {8^{1-p}}\min\left\{ |\Omega|^{2(1-p)}, |\Omega|^{1-p}, M_\Omega^{1-p}\right\}  [u]^p_{W^{s,1}(\Rn)},\]
  otherwise if
  \[    \gamma_{n,s,p} < 1  ,\]
  then
 \begin{equation*} \begin{aligned} [u]^p_{W^{s_p,p}(\Rn)} \geq &\;   \gamma^p_{n,s,p} {8^{1-p}}\min\left\{ |\Omega|^{2(1-p)}, |\Omega|^{1-p}, M_\Omega^{1-p}\right\}[u]^p_{W^{s,1}(\Rn)}
\\
\geq &\; 
 \left(e^{\frac{n}s+1}\right)^{1-p} \left(8s\omega_n\diam(\Omega)^s\right)^{1-p} \left(\min\left\{ |\Omega|^{2}, |\Omega| M_\Omega\right\}\right)^{1-p}[u]^p_{W^{s,1}(\Rn)} ,
\end{aligned}\end{equation*}
 counting  on the fact that
    \[ \frac{\log{p}}{p-1} \leq 1, \qquad \frac{\log\frac{s_p}{s}}{p-1} = \frac{\log\left(1+ \displaystyle\frac{n(p-1)}{sp}\right)}{p-1} \leq \frac{n}{s}.\]
    We have reached our conclusion.
 \end{proof}
 
The main result of the Section is the following.

\begin{theorem}\label{exmin}
Let $f\in L^{\frac{n}{s}}(\Omega)$ be such that
\begin{equation}  \begin{aligned}\label{q} \|f\|_{L^{\frac{n}{s}}(\Omega)}  \leq \frac{1}{ 2S_{n,s}},\end{aligned}\end{equation}
where
 $ S_{n,s}$ is given in  \eqref{sobct}. Let $\{u_p\}_p \in \W_0^{s_p,p}(\Omega)$ be a sequence of $(s_p,p)$-minimizers. Then, there exists $u_1 \in \W^{s,1}_0(\Omega)$ such that, up to a subsequence,
\begin{equation}  \begin{aligned} \label{kl} u_p \xrightarrow[p\to 1]{} u_1 \quad \mbox{ in } L^1(\Omega), \,  \mbox{  a.e. in }  \Rn \, \mbox{ and  weakly in } L^{\frac{n}{n-s}}(\Omega).\end{aligned}\end{equation}
 Furthermore, $u_1$ is an $(s,1)$-minimizer and weak solution of \eqref{defn1}.  
\\
If 
\[		\|f\|_{L^{\frac{n}{s}}(\Omega)}  < \frac{1}{ 2S_{n,s}},\]
 $u_1=0$ is the unique $(s,1)$-minimizer and a weak solution of \eqref{defn1}. 
\\ Furthermore, it holds that
 \begin{equation}  \begin{aligned}\label{limitss} \lim_{p\to 1} \F_p^{s_p}(u_p)= \F_1^s(u_1).\end{aligned}\end{equation}
\end{theorem}

\begin{proof} We articulate the proof in five parts. We first use Proposition \ref{ub} to get a uniform  bound (in $p$) of the $W^{s,1}(\Omega)$ norm of the $(s_p,p)$-minimizer and obtain, by compactness, the existence of $u_1$ in the limit as $p\to 1$. We then focus on showing that $u_1$ is  an $(s,1)$-minimizer and a weak solution of \eqref{mainpb}. We easily obtain the result that $u_1=0$ is the unique minimizer in the case of strict inequality. In the last part, we study the pointwise limit \eqref{limitss}. 
\medskip

\noindent \textbf{Part 1. Uniform bound on the $(s_p,p)$-minimizer.}\\
We apply Corollary \ref{expcor} and obtain that there exists a unique $(s_p,p)$-minimizer $u_p\in \W_0^{s_p,p}(\Omega)$. 
Comparing with the null function, we have that
\[ \F^{s_p}_p (u_p) \leq \F^{s_p}_p (0) =0,\]
hence using Proposition \ref{ub} and the Sobolev inequality,
\begin{equation}  \begin{aligned}  \label{obs2} \frac{1}{2p} [u_p]_{W^{s,1}(\Rn)}^p \leq  &\; \frac{C_{n,s,\Omega}^{1-p}}{2p}  [u_p]_{W^{s_p,p}(\Rn)}^p \leq C_{n,s,\Omega}^{1-p} \int_\Omega fu_p \, dx \\ \leq &\; C_{n,s,\Omega}^{1-p} S_{n,s}\|f\|_{L^{\frac{n}{s}}(\Omega)} [u_p]_{W^{s,1}(\Rn)}.\end{aligned}\end{equation}
We obtain
\[\, [u_p]_{W^{s,1}(\Rn)}^{p-1} \leq C_{n,s,\Omega}^{1-p} p\left(2  S_{n,s} \|f\|_{L^{\frac{n}{s}}(\Omega)}\right) ,\]
and
 \begin{equation}  \begin{aligned}\label{kl1111} [u_p]_{W^{s,1}(\Rn)}\leq \left(2  S_{n,s} \|f\|_{L^{\frac{n}{s}}(\Omega)}\right)^{\frac{1}{p-1}} \frac{e}{C_{n,s,\Omega}} .\end{aligned}\end{equation}
 By \eqref{equinorm}, we get that the sequence $u_p$ is uniformly bounded in the $W^{s.1}(\Omega)$ norm. 
 Since $L^1(\Omega) \subset W^{s,1}(\Omega)$ compactly, there exists $\tilde u_1\in W^{s,1}(\Omega)$  such that up to subsequences,
$u_p \to \tilde u_1$ as $p\to 1$, in $L^1$ norm and almost everywhere in $\Omega$. \\
Furthermore, $\|u_{p}\|_{L^{\frac{n}{n-s}}(\Omega)}$ is uniformly bounded (by the Sobolev inequality), hence up to subsequences  
\[ u_{p} \xrightarrow[p \to 1]{}  u_1 \]
weakly in $L^{\frac{n}{n-s}}(\Omega)$, i.e.
\begin{equation}  \begin{aligned}\label{wconv}\lim_{p\to 1} \int_\Omega f u_p\, dx= \int_\Omega f u_1\, dx.\end{aligned}\end{equation}
 We let $u_1=\tilde u_1 $ in $\Omega$ and $u_1=0$ in $\Co \Omega$ and get  $u_1\in \W^{s,1}_0(\Omega)$ with the desired properties.
\medskip

\noindent \textbf{Part 2. Existence of a minimizer}\\ 
We prove now that $u_1$ is an $(s,1)$-minimizer of $\mathcal F_1^s (u)$. Indeed, let $v\in \W^{s,1}_0(\Omega)$ be a competitor for $u$. By the density of $C^\infty_c(\Omega)$ in $\W^{s,1}(\Omega)$, we have that there exists a sequence $\{ \psi_j\}_{j\in \N} $ with
$\psi_j \colon \Rn \to \R$ such that 
\[ \psi_j \in C^\infty_c(\Omega), \qquad \lim_{j\to +\infty} \|v-\psi_j\|_{W^{s,1}(\Omega)}=0,\]
and $ \lim_{j \to +\infty} (v-\psi_j)= 0$ almost everywhere in $\Rn$.  
Notice that
\begin{equation}  \begin{aligned}\label{gr} \lim_{j \to +\infty} [\psi_j]_{W^{s,1}(\Rn)}= &\;[v]_{W^{s,1}(\Rn)},  \quad \lim_{j \to +\infty} \int_\Omega f \psi_j \, dx = \int_\Omega f v\, dx, \\ &\; \lim_{j \to +\infty} \Fi(\psi_j)=\Fi(v),
\end{aligned}\end{equation}
given that
 \begin{equation*} \begin{aligned} \left| \int_\Omega (f\psi_j - f v)\,dx \right| \leq & \; C  \|f\|_{L^{\frac{n}{s}}(\Omega)}\|\psi_j -v\|_{L^{\frac{n}{n-s}}(\Omega)}  
\\
\leq &\; C \|f\|_{L^{\frac{n}{s}}(\Omega)} [\psi_j - v]_{W^{s,1}(\Rn)}  \\
\leq  &\;C \|f\|_{L^{\frac{n}{s}}(\Omega)}  \|\psi_j - v\|_{W^{s,1}(\Omega)} 
\end{aligned}\end{equation*}
and 
\[
\left|[ \psi_j ]_{W^{s,1}(\Rn)}- [v ]_{W^{s,1}(\Rn)} \right| \leq  [ \psi_j -v ]_{W^{s,1}(\Rn)}  \leq  C \|\psi_j - v\|_{W^{s,1}(\Omega)}\]
by H{\"o}lder's, Sobolev's inequalities and \eqref{equinorm}. Here above $C$ denotes a positive constant, depending on $n,s,\Omega$, that may change value from line to line. 
According to Theorem \ref{thlim}, for all $j\in \N$,
\begin{equation}  \begin{aligned} \label{qwe} \lim_{p \to 1} \E^{s_p}_p(\psi_j) = \E^s_1(\psi_j).\end{aligned}\end{equation}
Using, in order, Fatou's lemma coupled with  \eqref{wconv}, the minimality of $u_p$, and \eqref{qwe},  we have the line of inequalities
\begin{equation}  \begin{aligned}\label{mininft}\F_1^s(u_1)\leq&\; \liminf_{p\to 1} \F_p^{s_p}(u_p)  \leq \lim_{p\to 1} \F_p^{s_p}(\psi_j) =  \F^s_1(\psi_j).\end{aligned}\end{equation}
Now from \eqref{gr},
\begin{equation}  \begin{aligned}\label{minu1} \F_1^s(u_1)- \F_1^s(v) \leq&\;   \lim_{j\to +\infty}\F^s_1(\psi_j)- \F_1^s(v) =0,\end{aligned}\end{equation}
and we obtain the desired conclusion
 that $u_1$ is a minimizer of $\F_1^s(u)$.

\bigskip
\noindent \textbf{Part 3. Existence of a weak solution}
We give here a sketch of the proof, following \cite[Theorem 3.4]{mazfr} (see also \cite[Theorem 1.6 (iii)]{bdlvm}), and leave the  details to these two references. 
We have   thanks to Theorem \ref{pexist} that  $u_p$ is a weak solution  of \eqref{mainpbp}, hence
\begin{equation}  \begin{aligned} \label{qwer}
\frac{1}{2}\int_{\R^{2n}} \frac{|u_p(x)-u_p(y)|^{p-2}(u_p(x)-u_p(y))(w(x)-w(y))}{|x-y|^{n+s_p p}} \, dx dy = \int_\Omega fw \, dx
\end{aligned}\end{equation}
for all $w\in \W^{s_p,p}_0(\Omega)$. We take a sequence $\{p_k\}_k $ with $p_k\to 1$ as $k\to +\infty$ and
we define
\[ C_{p_k,M} = \left\{ (x,y)\in \R^{2n} \, \bigg| \, \frac{|u_{p_k}(x)-u_{p_k}(y)|^{{p_k}-2} (u_{p_k}(x)-u_{p_k}(y))}{|x-y|^{n+s_{p_k}{p_k}}}>M
 \right\}\]
and, thanks to the uniform bound 
\[ \|u_{p_k}\|_{W^{s,1}(\Omega)} \leq c,\]
we are able to
show that there exists a subsequence ${p_k^M}$ and multi-valued function $\z\colon \R^{2n} \to [-1,1]$  such that
$\|	\z\|_{L^\infty(\R^{2n})} \leq 1$ and
denoting 
\[ U_{p_k^M} (x,y):= |u_{p_k^M}(x)-u_{p_k^M}(y)|^{{p_k^M}-2} (u_{p_k^M}(x)-u_{p_k^M}(y)),\]
it holds that 
 \begin{equation*} \begin{aligned}&\;  \lim_{M\to +\infty} \lim_{k \to +\infty} \int_{\R^{2n}}\frac{ U_{p_k^M} (x,y)(w(x)-w(y))}{|x-y|^{n+s_{p_k^M}{p_k^M}}} \chi_{\R^{2n} \setminus C_{{p_k^M},M} } (x,y) \, dx dy \\
=&\; \int_{\R^{2n}} \frac{\z(x,y)(w(x)-w(y))}{|x-y|^{n+s}}\, dx dy,
\end{aligned}\end{equation*}
together with
\begin{equation*} \begin{aligned}\lim_{M\to +\infty} \lim_{k \to +\infty} \int_{\R^{2n}}\frac{ U_{p_k^M} (x,y)(w(x)-w(y))}{|x-y|^{n+s_{p_k^M}{p_k^M}}} \chi_{C_{{p_k^M},M} } (x,y) \, dx dy =0,
 \end{aligned}\end{equation*}
for all $w\in \W_0^{s,1}(\Omega)$.
Therefore, using \eqref{qwer}
\[ \frac{1}{2}\int_{\R^{2n}} \frac{\z(x,y)(w(x)-w(y))}{|x-y|^{n+s}}\, dx dy=\int_\Omega fw\, dx,\]
for all $w\in \W_0^{s,1}(\Omega)$. At this point, it remains to see that $\z$ obtained in this way satisfies  
\begin{equation}  \begin{aligned} \label{cl} \z(x,y) \in \sgn(u_1(x)-u_1(y)).\end{aligned}\end{equation}
Indeed, taking $w=u_1$, we have that
 \begin{equation*} \begin{aligned} \frac{1}{2} \int_{\R^{2n}} \frac{\z(x,y)(u_1(x)-u_1(y))}{|x-y|^{n+s}}\, dx dy= &\; \int_\Omega fu_1\, dx 
\\ 
 \geq &\; \frac{1}{2} \int_{\R^{2n}} \frac{|u_1(x)-u_1(y)|}{|x-y|^{n+s}}\, dx dy,
  \end{aligned}\end{equation*}
given that $u_1$ is an $(s,1)$-minimizer and comparing with the null function. 
This shows \eqref{cl} and, according to Definition \ref{weaksol}, concludes the proof that $u_1$ is a weak solution. 

\medskip
\noindent \textbf{Part 4. Null minimizer/weak solution}.
Due to the strict inequality \[ \|f\|_{L^{\frac{n}s}(\Omega)}<(2S_{n,s})^{-1},\]
from \eqref{kl1111} by sending $p \to 1$ we obtain that 
\[ u_p \longrightarrow u_1=0. \]  Notice furthermore that using the Sobolev inequality,
\[
\F_1^s(v) \geq [ v]_{W^{s,1}(\Rn)}  \left(\frac12-S_{n,s} \|f\|_{L^{\frac{n}s}(\Omega)}\right) >0,\]
for all $v\in \W_0^{s,1}(\Omega)$, $v \neq 0$. Thus $u_1=0$ is the unique minimizer/weak solution. 

\medskip
\noindent \textbf{Part 5. Pointwise limit}.
 By the density of $C^\infty_c(\Omega)$ in $\W^{s,1}(\Omega)$, we have that there exists a sequence $\{ \phi_j\}_{j\in \N} $ with
$\phi_j \colon \Rn \to \R$ such that 
\[ \phi_j \in C^\infty_c(\Omega), \qquad \lim_{j\to +\infty} \|u-\phi_j\|_{W^{s,1}(\Omega)}=0. \] 
Notice that, as in \eqref{gr},
\[ \lim_{j\to +\infty} \F_1^s(\phi_j)= \F_1^s(u_1)\]
and reasoning as in \eqref{mininft}, we get that
\begin{equation*} \begin{aligned} \F_1^s(u_1) \leq &\; \liminf_{p\to 1} \F_p^{s_p}(u_p)\leq \limsup_{p\to 1} \F_p^{s_p}(u_p) \leq \lim_{j\to +\infty} \lim_{p \to1} \F_p^{s_p}(\phi_j) 
\\ =
&\;  \lim_{j\to +\infty} \F_1^{s}(\phi_j) = \F_1^s(u_1)
 \end{aligned}\end{equation*}
and the conclusion immediately follows.
\end{proof}
Regarding non-uniqueness of minimizers when \begin{equation} \label{kill} \|f\|_{L^{\frac{n}{s}}(\Omega)}  = (2S_{n,s})^{-1},\end{equation} we remark that in Section 5, precisely in Proposition \ref{zxcv2}, we give an example of a function (the constant function $1$), of a set (the ball $B_R$ for an $R$ that makes the equality \eqref{kill} hold) and for which all functions $u= \lambda \chi_{B_R}$ with $\lambda \geq 0$ are non-negative minimizers.

\begin{remark} \label{tech} We discuss shortly the technical reason for which the  specific fractional parameter $s_p$ in \eqref{sppar} is needed. 
As a matter of fact, with the methods here employed,  one cannot hope to work with the $(s,p)$-energy instead of the $(s_p,p)$-energy. The underlying  justification lies in the fact that instead of \eqref{obs1}, one would obtain that the $W^{s,p}(\Omega)$ semi-norm of $u_p$ is bounded from below by the $W^{\sigma,1}(\Omega)$ semi-norm of $u_p$ for some $\sigma \in (0,s)$ (see, in this regard, for instance \cite[Lemma 3.1]{bdlv}). On the other hand -- unless one requires higher integrability on $f$, i.e. that $f\in L^{\frac{n}{\sigma}}(\Omega)$ --  \eqref{obs2} would read 
\[ \frac{1}{2p} [u_p]_{W^{\sigma,1}(\Rn)}^p \leq  C S_{n,s}\|f\|_{L^{\frac{n}{s}}(\Omega)} [u_p]_{W^{s,1}(\Rn)}. \] 
This inequality however does not allow to continue -- since the $(\sigma,1)$ seminorm is smaller then the $(s,1)$-seminorm,  \cite[Proposition 2.1]{hitch}, and not vice-versa.
\end{remark}

\medskip

Let us point out a more general result about the equivalence between minimizers and weak solutions, as in \cite[Theorem 1.6]{bdlvm}.

\begin{proposition}\label{exmin1}
a) Let $u\in \W^{s,1}_{0}(\Omega)$ be a weak solution of \eqref{mainpb}. Then $u$ is a minimizer of $\F_1^s(u)$. \\
b) Suppose that here exists a weak solution of \eqref{mainpb}. Then any minimizer of $\F_1^s(u)$ is a weak solution of \eqref{mainpb}. 
\end{proposition}

\begin{proof}
a) If $u$ is  weak solution, consider $v\in \W_0^{s,1}(\Omega)$ any competitor for $u$, and use $w=u-v$ in \eqref{ee1}.
Using that $\z(x,y)(u(x)-u(y))=|u(x)-u(y)|$ and that $ \z(x,y)(v(x)-v(y)) \leq |v(x)-v(y)|$, we have 
\[\frac{1}{2} \int_{\R^{2n}} \frac{\z(x,y)(u(x)-u(y)-v(x)+v(y))}{|x-y|^{n+s}} \, dx dy = \int_\Omega f(u-v)\, dx,\]
hence
\begin{equation*} \begin{aligned}\F_1^s(u)= &\; \frac{1}{2}\int_{\R^{2n}} \frac{|u(x)-u(y)|}{|x-y|^{n+s}} \, dx dy - \int_\Omega fu \, dx 
\\
=&\; \frac{1}{2}\int_{\R^{2n}} \frac{\z(x,y)(v(x)-v(y))}{|x-y|^{n+s}} \, dx dy - \int_\Omega fv \, dx
\\
\leq &\; \frac{1}{2}\int_{\R^{2n}} \frac{|v(x)-v(y)|}{|x-y|^{n+s}} \, dx dy - \int_\Omega fv \, dx = \F_1^s(v),
\end{aligned}\end{equation*}
and we have the thesis. 
\\
b) If there exists $\bar u$ a weak solution, then consider any $w\in \W_0^{s,1}(\Omega)$ and it holds that there exists $\z \in L^\infty({\R^{2n}})$ such that $\z (x,y)\in \sgn(\bar u(x)-\bar u(y)$ and such that
\begin{equation}  \begin{aligned} \label{bn} 0 = \frac{1}{2}\int_{\R^{2n}} \frac{\z(x,y)(w(x)-w(y))}{|x-y|^{n+s}}\, dx dy -\int_{\Omega} fw\, dx .\end{aligned}\end{equation}
 We want to prove that $\z(x,y) \in \sgn(u(x)-u(y))$. Let us take 
 $w=\bar u - u$, and notice that
 \begin{equation*} \begin{aligned}0=&\; \frac{1}{2} \int_{\R^{2n}} \frac{\z(x,y)(\bar u(x)-\bar u(y))}{|x-y|^{n+s}}\, dx dy - \frac{1}{2}\int_{\R^{2n}} \frac{\z(x,y)( u(x)- u(y))}{|x-y|^{n+s}}\, dx dy \\ &\; -\int_{\Omega} f (\bar u -u)\, dx  
 \\
 \geq &\; \frac{1}{2}\int_{\R^{2n}} \frac{|\bar u(x)-\bar u(y)|}{|x-y|^{n+s}}\, dx dy -\int_{\Omega} f \bar u \, dx  - \frac{1}{2}\int_{\R^{2n}} \frac{| u(x) -u(y)|}{|x-y|^{n+s}}\, dx dy \\ &\;  + \int_\Omega fu \, dx
 \\
 \geq &\; 0.
\end{aligned}\end{equation*}
 Since by using \eqref{bn}  first with $w=\bar u$ and then  with $w=u$, 
we have that
\[ \frac{1}{2} \int_{\R^{2n}} \frac{| u(x)- u(y)|}{|x-y|^{n+s}}\, dx dy = \frac{1}{2}\int_{\R^{2n}} \frac{\z(x,y)( u(x)- u(y))}{|x-y|^{n+s}}\, dx dy, \]
 the conclusion follows.
\end{proof}

Putting together Theorem \ref{exmin} and Proposition \ref{exmin1}, we have that if 
$\|f\|_{L^{\frac{n}{s}}(\Omega)} \leq (2S_{n,s})^{-1}$ then $u$ is a minimizer if and only if $u$ is a weak solution, and the only minimizer/weak solution is the null function. The interesting case here remains 
$\|f\|_{L^{\frac{n}{s}}(\Omega)} = (2S_{n,s})^{-1}$, case in which non-null minimizers exist, and we give an example in the subsequent Section \ref{examples}. Precisely, in Proposition \ref{zxcv2}, for the constant function $1$ and a ball $B_R$ with $R$ such that the above equality holds, any function $u=\lambda \chi_{B_R} $ for $\lambda \geq 0$ is a minimizer.
 
\medskip

\begin{remark} We point out that any weak-solution is a local minimizer in the following sense. For all open sets $A \subset \subset \Omega$ we let
\[ E(u, A)= \iint_{Q(A)} \frac{|u(x)-u(y)|}{|x-y|^{n+s}} \, dx dy - \int_{A} f u \, dx,\] where we recall the notation $Q(A)= \R^{2n} \setminus (\Co A)^2$. We point out that $u$ only vanishes on $\Co \Omega$, so the first term in the energy $E(u, A)$ is not equal to $[u]_{W^{s,1}(\Rn)}.$  If $u\in \W^{s,1}_0(\Omega)$ is a weak solution in the sense of Definition \ref{weaksol} then \[ E(u,A)\leq E(v, A)\] for all competitors $v$, i.e. $v\in \W^{s,1}_0(\Omega)$ such that $v=u$ in $\Co A$. Indeed, for any competitor $v$ we consider $w=u-v$ and notice that $w=0$ in $\Co A$, hence 
\begin{equation*} \begin{aligned}
 & \int_{\R^{2n}} \z(x,y)\frac{w(x)-w(y)}{|x-y|^{n+s}} \, dx dy - \int_\Omega fw \, dx \\ = & \iint_{Q(A)}\z(x,y)\frac{w(x)-w(y)}{|x-y|^{n+s}} \, dx dy - \int_A f w \, dx.\end{aligned} \end{equation*}  Then, proceeding as in the proof of Proposition \ref{exmin1}, we obtain the result.\end{remark}

Even though a minimizer is not always unique, we can say something more about the "uniqueness" of the multi-valued function $\z$. 
Notice that in the proof of Proposition \ref{exmin1}, it comes up that if $\bar u$ is a weak solution and $\z$ is used to verify this solution, then the same $\z$ can be used to verify any other  weak solution.  This is a actually a general fact, true in the classical framework as well, see \cite[Remark 2.9]{robin}.

\begin{corollary}\label{vb}
Let  $u_1, u_2 \in \W^{s,1}_0(\Omega)$ be two weak solutions of \eqref{mainpb}, and let  $\z_1, \z_2$ be as in \eqref{ee1},
 \eqref{ee2}. 
Then almost everywhere in $Q(\Omega)$,
\[ \z_j(x,y) \in \sgn(u_i(x)-u_i(y))\]
for $i,j \in \{1,2\}.$
\end{corollary}

\begin{proof}
We observe that, since for all  $ i\neq j\in \{ 1,2\}$
	\begin{equation*}
	\begin{aligned} &\z_i(x,y) (u_i(x)-u_i(y))=|u_i(x)-u_i(y)|, \\ & \z_i(x,y) (u_j(x)-u_j(y)) \leq |u_j(x)-u_j(y) |,
	\end{aligned}
	\end{equation*}
 we have that 
\begin{equation*} \begin{aligned}
 \z_1(x,y)\left(u_1(x) - u_1(y) - (u_2(x)-u_2(y))\right)  =&\; |u_1(x)-u_1(y)| -   \z_1(x,y)((u_2(x)-u_2(y)) \\
\geq &\;  |u_1(x)-u_1(y)|  - |u_2(x)-u_2(y)|  
\end{aligned}\end{equation*}
and
\begin{equation*} \begin{aligned}
\z_2(x,y)\left(u_1(x) - u_1(y) - (u_2(x)-u_2(y))\right)  = &\;  \z_2(x,y)((u_1(x)-u_1(y)) -  |u_2(x)-u_2(y)|   \\ \leq &\;  |u_1(x)-u_1(y)|  -|u_2(x)-u_2(y)|.
	\end{aligned}\end{equation*}
	By hypothesis, we have that  for all $w \in \W^{s,1}_0(\Omega)$,
	\[
	  \iint_{\R^{2n})} \frac{\z_1(x,y)\left(w(x) - w(y)\right)}{ | x - y |^{n+s}}  dx\, dy  =\iint_{\R^{2n})} \frac{\z_2(x,y)\left(w(x) - w(y)\right)}{ | x - y |^{n+s}}  dx\, dy  .
\]
 Inserting the above inequalities for $w=u_1-u_2$, we have that
\begin{equation*} \begin{aligned}
  \iint_{\R^{2n}} \frac{ |u_1(x)-u_1(y)|  - |u_2(x)-u_2(y)| }{ | x - y |^{n+s}}  dx\, dy 
\leq &\; \iint_{\R^{2n})} \frac{\z_1(x,y)\left((u_1-u_2)(x) - (u_1-u_2)(y)\right)}{ | x - y |^{n+s}}  dx\, dy \\
  =&\; \iint_{\R^{2n})} \frac{\z_2(x,y)\left((u_1-u_2)(x) - (u_1-u_2)(y)\right)}{ | x - y |^{n+s}}  dx\, dy  
  \\ 
  \leq &\;  \iint_{\R^{2n}} \frac{  |u_1(x)-u_1(y)|  -|u_2(x)-u_2(y)| }{ | x - y |^{n+s}}  dx\, dy.
 \end{aligned}\end{equation*}	 
 It follows that almost everywhere in $Q(\Omega)$
\[|u_i(x)-u_i(y)| -   \z_j(x,y)((u_i(x)-u_i(y))=0 \] 
for $i,  j\in \{1,2\}$, and we have achieved our conclusion.
\end{proof}

\section{Necessary and sufficient conditions for existence when $f	$ is non-negative} \label{minper}

We provide in this section necessary and sufficient conditions for the existence of non-negative minimizers of $\F_1^s$, when $f$ is non-negative.  We further provide a sharp result on the asymptotics as $p\to 1$ of solutions of \eqref{mainpbp}. These sharp conditions are given in terms of the "weighted Cheegar constant". 

\noindent We focus on the case 
\begin{equation}  \begin{aligned} \label{ro1} f \geq 0, \qquad f\in L^{\frac{n}{\sigma}}(\Omega)  \end{aligned}\end{equation}
for some $\sigma\in (0,s),$
and  such that there is $ r_o>0 $ and  $x_o\in \Omega$ such that $B_{r_0}:=B_{r_o}(x_o)\subset \Omega$  and 
\begin{equation}  \begin{aligned} \label{ro} \int_{B_{r_o}} f(x) \, dx>0.\end{aligned}\end{equation}
Alternative conditions are $f\in L^\infty(\Omega)$  with $f>0$ almost everywhere in $\Omega$. Each of these alternatives includes  the case $f=1$, treated in \cite{bueno} in the classical case.  We remark that we are not able to cover the entire class $f\in L^{\frac{n}{s}}(\Omega)$ (see Remark \ref{bummers}). We point out that our result is new, even in the constant case $f=1$.

We first notice that for $f\geq 0$, if $u\in \W_0^{s,1}(\Omega)$   then $\F_1^s(u_+)\leq \F_1^s(u)$ with $u_+=\min\{u,0\}$, that is, minimizers are non-negative. It is not restrictive then to look for non-negative minimizers of the energy, i.e.
\[\Fi(u) \leq \Fi(v) \qquad \mbox{ for all   } v\in \W^{s,1}_0(\Omega) \; \mbox{ such that} \; v\geq 0
.\]
We observe a very important homogeneity feature of our energy. 
\begin{remark}\label{inf}
If there exists $v\in \W_0^{s,1}(\Omega)$ such that $\F_1^s(v)<0$, then for any $\lambda>0$,
\[\lim_{\lambda\to +\infty} \F_1^s(\lambda v)= \lim_{\lambda\to +\infty}  \lambda \F_1^s(v) =-\infty,\]
i.e. our energy is unbounded from below in the space $\W^{s.1}_0(\Omega)$, and a global minimizer does not exist. 
Therefore, to ensure the existence of a minimizer, it has to hold that
\[ \F_1^s(v) \geq 0 \qquad \mbox{ for all } \; v\in \W^{s,1}_0(\Omega).\] 
It follows that, if $\Omega $ is such that  there exists a minimizer in $\W^{s,1}_0(\Omega)$ then
\[ \F_1^s(u)=0\]
by comparing with the null function.
Moreover, if $u\neq 0$, then for all $\lambda \in \R$
\[ \Fi(\lambda u)=\lambda \Fi(u)=0,\]
and $\Fi$ has an infinite number of minimizers.

\end{remark}

To obtain  the results in this section, we establish a connection between non-negative minimizers of $\F_1^s$ and sets $E\subset \Omega$  that minimize
\[ \P(E) := \Per_s(E)-|E|_f,\]
where
we recall  the fractional perimeter in \eqref{fracp},  that 
\[ |E|_f=\int_E f\, dx,\] 
and that 
\begin{equation}  \begin{aligned} \label{yu} \F_1^s(\chi_E)=\P(E).\end{aligned}\end{equation}
We also point out that by Tonelli, for any non-negative function $u$,
 \begin{equation*} \begin{aligned} \int_\Omega f (x)u(x) \, dx =&\;  \int_\Omega  f(x)\left( \int_{(0,+\infty)} \chi_{\{u(x)\geq t\}}(t) \, dt\right) dx
	\\
	&\; =  \int_{(0,+\infty)} \left(\int_\Omega  f(x)\chi_{\{u\geq t\}}(x)\, dx\right) dt \\ = &\;\int_{(0,+\infty)} |\{u\geq t\}|_f \, dt,
	\end{aligned}\end{equation*}		
and that by the co-area formula \cite{visintin}, for $u\in \W_0^{s,1}(\Omega)$ it holds
\[ \,\frac{1}{2} [u]_{W^{s,1}(\Rn)}= \int_{(0,+\infty)}  \Per_s(\{u\geq t\})\,dt,\]
hence
\begin{equation}  \begin{aligned} \label{coar} 
\Fi(u)= \int_{(0,+\infty)} \P(\{u\geq t\})\, dt.\end{aligned}\end{equation}

We make some notes on 
 sets $E\subset \Omega$ of finite fractional perimeter that realize 
\[\P(E)
\leq \P(F) \quad \mbox{ for all } F\subset \Omega
,\]
proving at first existence for all $\Omega \subset \Rn$. 

\begin{proposition} Let $f\geq  0$, $f\in L^1(\Omega)$. For any bounded open set $\Omega \subset \Rn$, 
there exists $E\subset \Omega$ minimizer of $\P$. Furthermore, $E$ is of finite fractional perimeter.
\end{proposition}	
\begin{proof} We proceed by direct methods. 
We have that
\[ \P(E) \geq -\|f\|_{L^1(\Omega)} ,\] 
so there exists a minimizing sequence
\[ m:= \inf_{F\subset \Omega} \P(F) =\lim_{k\to +\infty} \P(E_k)\] and notice that $m+\|f\|_{L^1(\Omega)} \geq 0$. 
There exists some $\bar k>0$ such that for all $k\ \geq \bar k$ 
\[ \P(E_k) \leq m+1,\]
 hence
\[ \Per_s(E_k) \leq m+1+\|f\|_{L^1(\Omega)} \]
 thus by \eqref{glows1} for all $B_R \supset \supset \Omega$,
 \[ \|\chi_{E_k}\|_{W^{s,1}(B_R)} \leq C(m+1+ \|f\|_{L^1(\Omega)}) ,\]
with $C>0$, independent of $k$. By compactness of $W^{s,1}(B_R)$ in $L^1(B_R)$, and recalling that $E_k\cap \Co \Omega=\emptyset$
there exists some set $E\subset \Omega$ of finite $s$-perimeter  such that up to subsequences
\[\chi_{E_k} \xrightarrow[k \to +\infty]{} \chi_E \mbox{ in } L^1(\Omega) \mbox{ and  a.e. in } \Rn.\] 
By Fatou's lemma  and the dominated convergence theorem,
\[m\geq \P(E),\]
hence $E$ is a minimizer of $\P$. Also, there is some $B_{r_o}\subset \Omega$, and 
 \begin{equation*} \begin{aligned}& \Per_s(E) - |\Omega|_f\leq \Per_s(E) - |E|_f= \P(E)\leq \P(B_{r_o})\\ &\; = \Per_s(B_{r_o})- |B_{r_o}|_f 
 \leq \Per_s(B_{r_o})<+\infty.
 \end{aligned}\end{equation*}
This concludes the proof of the proposition. 
\end{proof}

Notice that if there exists some set $E \subset \Omega$ that minimizes $\P$ such that 
$ \P(E)<0$ 
then from \eqref{yu} and Remark \ref{inf},
$\inf_{u\in \W_0^{s,1}(\Omega)} \Fi(u) =-\infty.$
We  focus on the connection between minimal sets of $\P$ and minimizers of $\Fi$.

\begin{proposition} \label{poi} Let $f\geq  0$, $f\in L^1(\Omega)$.\\ 
(i) Let $E\subset \Omega$ be such that $\chi_E \in \W_0^{s,1}(\Omega)$ is a minimizer of $\Fi$, then $E$ is a minimizer of $\P$.\\
(ii) 
Let $E\subset \Omega $ be a minimizer of $\P$ such that 
$\P(E)\geq 0.$  Then $\chi_E$ is a non-negative minimizer of $\Fi$.
\end{proposition}
\begin{proof}
(i) For all $F\subset \Omega$ -- and we assume without loss of generality that $F$ has finite fractional perimeter --, we have 
\[ \P(E)=\Fi(\chi_E)\leq \Fi(\chi_F) =\P(F).\]
(ii) For any non-negative competitor $v\in \W_0^{s,1}(\Omega)$, we notice that by \cite[Lemma 2.8]{bdlv}, the set $\{v\geq t\}$ is of finite fractional perimeter.  Furthermore, we have that 
\[ \{ v\geq t\} \subset \Omega \qquad \mbox{ for all } t\in (0,+\infty),\] 
hence
 $\{ v\geq t\} $ is a competitor for $E$ for all $t\in (0,+\infty)$. By the minimality of $E$
\[ 0\leq \P(E) \leq \P(\{v\geq t\}) \qquad \mbox{ for all } t\in (0,+\infty).\]  
Now, 
if $t\in (0,1]$ then 
 $ \{ \chi_E \geq t\} =E ,$  while if  $t\in (1,+\infty)$ then  $ \{ \chi_E \geq t\} =\emptyset ,$ hence
\[ \P(\{ \chi_E \geq t\})   \leq \P(\{v\geq t\}) \qquad \mbox{ for all } t\in (0,+\infty).\]
Then by \eqref{coar} 
\[ \Fi(\chi_E)
= \int_{(0,+\infty) } \P (\{ \chi_E\geq t\})\,dt 
\leq \int_{(0,+\infty) }\P(\{v\geq t\})\, dt = \Fi(v).
\]
This concludes the proof.
\end{proof}

We recall the definition of the $(s,f)$-Cheegar constant in \eqref{cheeg}. From \cite[Proposition 5.3]{brch} -- with minor modifications -- we have the existence of an $(s,f)$-Cheegar set. 
We insert the proof for completeness, and to clarify the role played by $f$ in the proof of existence. 

\begin{proposition} \label{chegex} Let $f$ satisfy \eqref{ro1} and \eqref{ro}. 
For any bounded open set $\Omega \subset \Rn$ there exists  a $(s,f)$-Cheegar  set $\tilde E\subset \Omega$
such that
\[  h^f_s(\Omega):= \frac{\Per_s(\tilde E)}{|\tilde E|_f}. \]
Moreover, $h_s^f(\Omega), |\tilde E |_f, \Per_s(\tilde E)\in (0,+\infty)$. 
\end{proposition}
\begin{proof} Notice that for all $E\subset \Omega$, 
\[|E|_f \leq |\Omega|_f \leq \|f\|_{L^{\frac{n}{\sigma}}(\Omega)} |\Omega|^{\frac{n-\sigma}{n}}.\] 
From \eqref{ro}, there is $B_{r_o}\subset \Omega$, then $\Per_s(B_{r_o})\in (0,+\infty)$ and $|B_{r_o}|_f \in (0, |\Omega|_f)$, hence
\[ h_s^f(\Omega)<+\infty.\] 
We let $\{E_k\}_k$ be a minimizing sequence, i.e.
\[ \lim_{k\to +\infty}\frac{\Per_s(E_k)}{|E_k|_f} =  h_s^f(\Omega)<+\infty,\] 
with 
$|E_k|_f>0$ for all $k\in \N$. Then
 there is some $\bar k\in \N$ such that for all $k\geq \bar k$,
 \[ \Per_s(E_k) \leq  (h_s^f(\Omega)+1){|E_k|_f}\leq ( h_s^f(\Omega)+1)|\Omega|_f.\]
 By the Sobolev inequality
 \begin{equation}  \begin{aligned}\label{bum} |E_k|_f \leq \|f\|_{L^{\frac{n}{\sigma}}(\Omega)} |E_k|^{\frac{n-\sigma}{n}},\end{aligned}\end{equation}
and by   the isoperimetric inequality we obtain
\[ \left(\frac{|E_k|_f}{ \|f\|_{L^{\frac{n}{\sigma}}(\Omega)} }\right)^{\frac{n-s}{n-\sigma}}\leq  |E_k|^{\frac{n-s}n}  \leq 2S_{n,s} \Per_s(E_k) \leq 2S_{n,s }( h_s^f(\Omega)+1)|E_k|_f 
.
\]
This yields that
\begin{equation}  \begin{aligned} \label{fb} |E_k|_f\geq c_{n,s,\sigma,\Omega}>0.\end{aligned}\end{equation}
Using \eqref{glows1} for all $B_R \supset \supset \Omega$, by compactness of $W^{s,1}(B_R)$ in $L^1(B_R)$, and recalling that $E_k\cap \Co \Omega=\emptyset$,
there is $\tilde E\subset \Omega$ such that $\chi_{E_k} \to \chi_{\tilde E} $ in $L^1(\Omega) $ and almost everywhere in $\Rn$. By \eqref{fb} and the dominated convergence theorem we have that 
\[\lim_{k\to +\infty}|E_k|_f = |\tilde E|_f>0.\]
Also by Fatou's lemma  
\[m\geq \P(\tilde E)\] and $\tilde E$ is the minimizing set. Notice also that by the isoperimetric inequality and the fact that $h_s^f(\Omega)<+\infty$ we get that $\Per_s(\tilde E)\in (0,+\infty)$, and finally that $h_s^f(\Omega) \in (0,+\infty)$.
\end{proof}

\begin{remark} \label{bummers}
Notice the necessity of taking $f\in L^{\frac{n}{\sigma}}(\Omega)$ instead of $f\in L^{\frac{n}{s}}(\Omega)$ in \eqref{bum}. The larger class would not allow to obtain a uniform lower bound for $|E_k|_f$.
\end{remark}

\begin{remark}
When $f\in L^\infty(\Omega)$ with $f>0$ almost everywhere, one could use also the definition
\[ h^f_s(\Omega):=\inf \left\{ \frac{\Per_s(A)}{|A|_f}\, \bigg| \,\,  A\subset \Omega, \, |A|>0\right\}. \]
It holds that $|E|_f>0$ for all $|E|>0$, while the uniform bound \eqref{fb} follows by using
\[ |E_k|_f\leq \|f\|_{L^\infty(\Omega)} |E_k|\]
instead of \eqref{bum}.  
\end{remark}
Notice furthermore that for all $E\subset \Omega$,
\[ \Per_s(E) - h^f_s(\Omega) |E|_f \geq 0, \qquad \mbox{ 
and }
\qquad  \Per_s(\tilde E) - h_s^f(\Omega) |\tilde E|_f=0,\]
hence $\tilde E$ - the $(s,f)$-Cheegar set, minimizes $\Per_s(E) - h^f_s(\Omega) |E|_f$ among all sets contained in $\Omega$.

\medskip 
Our sharp existence result if the following.
\begin{theorem} \label{fig} Let $f$ satisfy \eqref{ro1} and \eqref{ro}. \\
If $h_s^f(\Omega)  >1$ then the null function is the unique minimizer of $\Fi$. \\
If  $h_s^f(\Omega)  =1$ there exists a non-negative minimizer of $\Fi$.\\
If $h_s^f(\Omega) < 1$ then 
\[ \inf_{u\in \W_0^{s,1}(\Omega)} \Fi (u)=-\infty.\] 
\end{theorem}
\begin{proof}
If $h_s^f(\Omega) \geq 1$, then
\[ \P(A) = \Per_s(A) -|A| \geq \Per_s(A) - h^f_s(\Omega) |A|\geq 0.\]  
for all $A\subset \Omega$. Let $E$ be a minimizer of $\P$, then $\P(E)\geq 0$. According to Proposition \ref{poi} (ii), the function $\chi_E$ is a minimizer of $\Fi$. Hence, if the inequality is strict, $E=\emptyset$ is the unique minimal set of $\P$, hence the unique minimizer of $\Fi$ is the null function, otherwise any $\lambda \chi_E$ is a minimizer (see Remark \ref{inf}). 
\\
If $h_s^f(\Omega)<1$, then 
\[ \Fi(\chi_{\tilde E})= \Per_s(\tilde E) -|\tilde E|_f<\Per_s(\tilde E) -h_s^f(\Omega) |\tilde E|_f=0,\] 
where $\tilde E$ is an $(s,f)$-Cheegar set. 
The conclusion follows from Remark \ref{inf}.
\end{proof}

We complement the existence result with the study of the asymptotics. Notice that from Corollary \ref{expcor}, there is a unique minimizer of $\F_p^{s_p}$, since $n/\sigma > n/s$.
\begin{theorem}\label{ass}
Let $f$ satisfy \eqref{ro1}, \eqref{ro}  and $u_p$ be the unique minimizer of $\F_p^{s_p}$. \\
If $h^f_s(\Omega) \geq 1$ then 
\[ u_p \xrightarrow[p\to 1]{} u_1,
\]
in $L^1(\Omega)$ and almost everywhere in $\Rn$,
where $u_1$ is a minimizer of $\F_1^s$. Furthermore, $u_1$ is a weak solution of \eqref{mainpb}.\\
If $h_s(\Omega) < 1$ then 
\[ \lim_{p\to 1} [u_p]_{W^{s_p,p}(\Rn)} =+\infty\]
\end{theorem}

\begin{proof}

We look first at the case $h_s^f(\Omega)\geq 1$.
Notice that $u_p \in \W_0^{s_p,p}(\Rn)$ implies that $u_p\in \W_0^{s,1}(\Rn)$ (see \eqref{of3}) hence  for all $t\in (0,+\infty)$, 
$\{u_p \geq t \} $ is of finite fractional perimeter. 
By the definition of the Cheegar constant  and the existence of a $(s,f)$-Cheegar set we have
\[ \Per_s(\{u_p \geq t \}) \geq {h_s^f(\Omega)} |\{u_p \geq t \}|_f \]
and integrating, by the co-area formula \eqref{coar}, using that $u_p$ minimizes $\F_p^{s_p}$ (so $\F_p^{s_p}(u_p)\leq 0$) and Proposition \ref{ub},
\begin{equation*} \begin{aligned}\frac{1}2 [u_p]_{W^{s,1}(\Rn)} \geq &\; h_s^f(\Omega) \int_{\Omega} f u_p \, dx  \geq  {h_s^f(\Omega)} \frac{1}{2p} [u_p]_{W^{s_p,p}(\Rn)}^p \\
\geq  &\; {h_s^f(\Omega)} \frac{1}{2p} [u_p]_{W^{s,1}(\Rn)}^p C_{n,s,\Omega}^{p-1} .
 \end{aligned}\end{equation*}
Hence
\begin{equation}  \begin{aligned}\label{opp} [u_p]_{W^{s,1}(\Rn)}\leq \left(\frac{p} {h_s^f(\Omega)}\right)^{\frac{1}{p-1}} \frac{1}{C_{n,s,\Omega}} \leq \frac{e}{C_{n,s,\Omega}},\end{aligned}\end{equation}
and from \eqref{equinorm},
\[ \|u_p\|_{W^{s,1}(\Omega)} \leq C\]  with $C>0$ independent of $p$. By compactness, 
and reasoning as  in \eqref{mininft}
and \eqref{minu1} we obtain that $u_p$ converges to $u_1$, minimizer of $\F_1^s$. The fact that $u_1$ is also a weak solution follows as in Part 3 of the proof of Theorem \ref{exmin}. That $u_1=0$ when $h_s^f(\Omega)>1$ is clear either from Theorem \ref{fig} or sending $p\to 1$ in \eqref{opp}.
Observe also that we obtain
\begin{equation}  \begin{aligned} \label{inf1} \limsup_{p\to 1} [u_p]^{p-1}_{W^{s_p,p}(\Rn)} \leq \frac{1}{h_s^f(\Omega)}.\end{aligned}\end{equation}

We consider now the case $h_s^f(\Omega)<1$. Denote $\tilde E$ an $(s,f)$-Cheegar set. 
We first notice, as in \cite[Remark 5.4]{brch}, that $\partial \tilde E$ has to touch the boundary  of $ \Omega$. Let $\{t_k\}_{ k\in \N}$ be a sequence such that $t_k \to 1^-$ as $ n \to +\infty$ and let $E_k =t_k \tilde E$. Then $\overline{E_k } \subset \Omega$, $\chi_{E_k}\to \chi_{E}$ as $k \to +\infty$ and by the dominated convergence theorem
\[ \lim_{k\to \infty} |E_k|_f=|\tilde E|_f.\]
 Since $|\tilde E|_f  \in (0,|\Omega|_f]$ from Proposition \ref{chegex}, for $k$ large enough
$|E_k|_f\in (0,|\Omega|_f+1).$
We also have that
\begin{equation}  \begin{aligned} \label{limk} \frac{\Per_s(E_k)}{|E_k|_f} = t_k^{-s}\frac{\Per_s(\tilde E)}{|\tilde E|_f} \xrightarrow[k \to +\infty]{} \frac{\Per_s(\tilde E)}{|\tilde E|_f} =h_s^f(\Omega) \in (0,+\infty) ,\end{aligned}\end{equation}
for $k$ large enough, hence also $ \Per_s(E_k)\in (0,+\infty),$ thus $\chi_{E_k} \in \W_0^{s,1}(\Omega)$.  
By density, for any $k$ large enough, there exists $\{\varphi_j^k\}_{j\in \N} \in C_0^\infty(\Omega)$ such that
\[ \, \lim_{j\to +\infty} \|\varphi_j^k -\chi_{E_k}\|_{W^{s,1}(\Omega)} =0.\] It also follows that
\begin{equation}  \begin{aligned} \label{limj} &\; \lim_{j\to +\infty}\int_\Omega f \varphi_j^k \, dx = |E_k|_f\in (0,|\Omega|_f+1), \\ &\;  \lim_{j\to +\infty}\frac{1}2 [\varphi_j^k]_{W^{s,1}(\Rn)} = \Per_s(E_k)\in (0,+\infty),\end{aligned}\end{equation}
reasoning as in \eqref{gr}. Hence for $k, j$ large enough
\[ \int_\Omega f\varphi_j^k \, dx\in (0,+\infty) ,\qquad  [\varphi_j^k]_{W^{s,1}(\Rn)} \in (0,+\infty).\] 
From Proposition \ref{ub} and the fact that $\varphi_j^k\in C_0^\infty(\Omega)$, also $[\varphi_j^k]_{W^{s_p,p}(\Rn)} \in (0,+\infty)$. 
We define, following an idea from \cite[Lemma 1]{bueno2}
\begin{equation}  \begin{aligned}\label{theconst} c_{p,j,k}^{p-1} =\left(2p-\frac{1}{j}\right)\frac{\displaystyle \int_\Omega f\varphi_j^k \, dx}{[\varphi_j^k]_{W^{s_p,p}(\Rn)}^p}\in (0,+\infty).\end{aligned}\end{equation}
Now, let $u_p$ be the minimizer of $\F_p^{s_p}$ and weak solution of \eqref{mainpbp}. Then
\[ \int_\Omega f u_p\,dx = \frac{1}{2} [u_p]_{W^{s_p,p}(\Rn)}^p,\] 
hence
\begin{equation*} \begin{aligned} &\; \frac{1}{2p} [u_p]_{W^{s_p,p}(\Rn)}^p -\int_\Omega f u_p\, dx = \frac{1-p}{2p} [u_p]_{W^{s_p,p}(\Rn)}^p
\leq \F_p^{s_p}(\phi)
\\
 =&\;  \frac{1}{p} \left( \frac{1}{2}[\phi]_{W^{s_p,p}(\Rn)}^p -p \int_\Omega f \phi \, dx\right)
\end{aligned}\end{equation*}
for all $\phi \in C_0^\infty(\Omega)$, by minimality. 
This yields that
\begin{equation*} \begin{aligned} [u_p]_{W^{s_p,p}(\Rn)}^p \geq \frac{1}{p-1} \left(2p \int_\Omega f\phi \, dx - [\phi]_{W^{s_p,p}(\Rn)}^p\right).\end{aligned}\end{equation*}
Let 
$ \phi = c_{p,j,k} \varphi_j^k \in C_0^\infty(\Omega).$ Then
\begin{equation*} \begin{aligned}[u_p]_{W^{s_p,p}(\Rn)}^p \geq  &\;  \frac{c_{p,j,k}}{p-1} \left(2p \int_\Omega f\varphi_j^k \, dx - c_{p,j,k}^{p-1}[\varphi_j^k]_{W^{s_p,p}(\Rn)}^p\right)
\\=&\;  \frac{c_{p,j,k}}{j(p-1)} \int_\Omega f\varphi_{j,k} \, dx.
\end{aligned}\end{equation*}
It follows that
\begin{equation*} \begin{aligned}
[u_p]_{W^{s_p,p}(\Rn)}^{p-1} \geq    c^{\frac{p-1}{p}}_{p,j,k} \left(\frac{1}{j}\int_\Omega f\varphi_{j,k} \, dx\right)^{\frac{p-1}{p}} \left(\frac{1}{p-1}\right)^{\frac{p-1}{p}},
\end{aligned}\end{equation*}
so
\begin{equation*} \begin{aligned}\liminf_{p \to 1} [u_p]_{W^{s_p,p}(\Rn)}^{p-1} \geq \lim_{p\to 1}c_{p,j,k}^{p-1} =\left(2-\frac{1}{j}\right)\frac{\displaystyle \int_\Omega f\varphi_j^k \, dx}{[\varphi_j^k]_{W^{s,1}(\Rn)}},
\end{aligned}\end{equation*}
where we have used \eqref{of2}. 
We send first $j\to \infty$ and then $k \to \infty$, make use of \eqref{limj} and \eqref{limk} and get that
 \begin{equation*} \begin{aligned}\liminf_{p \to 1} [u_p]_{W^{s_p,p}(\Rn)}^{p-1} \geq  \frac{|\tilde E|_f} {\Per_s(\tilde E)}= \frac{1}{h_s^f(\Omega)},
 \end{aligned}\end{equation*}
 and together with \eqref{inf1}, 
  \begin{equation*} \begin{aligned}\lim_{p \to 1} [u_p]_{W^{s_p,p}(\Rn)}^{p-1} = \frac{1}{h_s^f(\Omega)}.
  \end{aligned}\end{equation*}
For an arbitrary small enough (and such that) $\varepsilon\in (0, 1-h_s^f(\Omega))$ there is some $\bar p$ close to $1$ such that  for all $p\in (1,\bar p)$,
\[  [u_p]_{W^{s_p,p}(\Rn)} > \left(h_s^f(\Omega) +\varepsilon\right)^{\frac{1}{1-p}}, \]
hence 
\[\lim_{p \to 1}   [u_p]_{W^{s_p,p}(\Rn)}  =+\infty\]
and we conclude  the proof.  
\end{proof}

We remark that the limit to infinity obtained when $h_s^f(\Omega)<1$ excludes that $u_p$ might tend to a $W^{s,1}$ function. 

\begin{corollary} Let $f$ satisfy \eqref{ro1}, \eqref{ro}  and $u_p$ be the unique minimizer of $\F_p^{s_p}$.
There does not exist $u_1\in \W_0^{s,1}(\Omega)$ such that 
\[ u_p \xrightarrow[p\to 1]{} u_1 \quad \mbox{ in } L^1(\Omega)  \quad  \mbox{ and weakly in }  \, L^{\frac{n}{n-\sigma}}(\Omega).\]
\end{corollary} 

\begin{proof}
Suppose by contradiction that such a function $u_1$ exists. Employing \eqref{limitss}, and noticing that
\[\lim_{p\to 1} \int_{\Omega } fu_p \, dx = \int_\Omega fu_1\, dx\] 
we obtain
\[  \lim_{p\to 1}[u_p]_{W^{s_p,p}(\Omega)} =[u_1]_{W^{s,1}(\Omega)}.\] 
The contradiction immediately follows. 
\end{proof}
\medskip

\begin{corollary}
Let $f$ satisfy \eqref{ro1}, \eqref{ro} and 
\[ h_s^f(\Omega) \geq 1.\] Let $E\subset \Omega$ be a minimizer of $\P$ such that $\P(E)\geq 0$.
 Then in a weak sense
\[(-\Delta)_1^s \chi_E=f.\] 
\end{corollary}
\begin{proof}
In our hypothesis, Theorem \ref{ass} gives the existence of a weak solution of \eqref{mainpb}. Then any minimizer is also a weak solution according to  Proposition \ref{exmin1}. That $\chi_E$ is a weak solution follows from Proposition \ref{poi} ii).
\end{proof}
\medskip 
We provide contextually a characterization of the weighted fractional Cheegar constant, similar to that of \cite{cheeg}.
\begin{corollary}
Let $f$ satisfy \eqref{ro1}, \eqref{ro}  and $u_p$ be the unique minimizer of $\F_p^{s_p}$. 
Then 
\begin{equation*} \begin{aligned}\lim_{p \to 1} [u_p]_{W^{s_p,p}(\Rn)}^{p-1} = \frac{1}{h_s^f(\Omega)}
\end{aligned}\end{equation*}
and
\begin{equation*} \begin{aligned} \lim_{p \to 1} \left(\int_\Omega fu_p \,dx\right)^{p-1}= \frac{1}{h_s^f(\Omega)}.
\end{aligned}\end{equation*}
\end{corollary}

\begin{proof}
We have reached the first thesis in the proof of Theorem \ref{ass}. For the second thesis,  notice that $u_p$ is also a weak solution, thus
\[ \int_\Omega fu_p \, dx = \frac{1}{2} [u_p]^p_{W^{s_p,p}(\Omega)} .\qquad \qquad \qedhere\] 
\end{proof}

 \smallskip 

\noindent As further observation, we have the following result, mimicking \cite[Theorem 1.3]{bdlv} and saying basically  -- under some additional assumptions -- that $u$ is a minimizer of $\Fi$ if and only if every level set of $u$ is a minimizer for $\P$.
\begin{proposition}
(i) Let $h_s^f(\Omega)\geq 1$ and let $u\in \W_0^{s,1}(\Omega)$. If for almost all $t\in(0,+\infty)$, the set $\{u\geq t\}$ is a minimizer of $\P$
 then $u$ is a non-negative minimizer of $\Fi$.
 \\
 (ii) There exists a weak solution of $\Fi$. If $u\in \W_0^{s,1}(\Omega)$ is a non-negative minimizer of $\Fi$, the set $\{u\geq t\}$ is a minimizer of $\P$.
\end{proposition}

\begin{proof}
(i) Let
\[ Z:=\{ t\in (0,+\infty) \, | \, \mbox{ the set $\{u\geq t\}$ is a minimizer of $\P$} \}\]
and it holds that $\mathcal L^1(Z)=0$.  We have that, for any $v\in \W^{s,1}_0(\Omega)$ and $t\in (0,+\infty)\setminus Z$,
\[ 0\leq \P(\{u\geq t\}) \leq \P(\{v\geq t\}),\]
where the non-negativity follows from the existence of a minimizer of $\Fi$. 
By the coarea formula
\begin{equation*} \begin{aligned} & \Fi(u)= \int_{(0,+\infty)} \P(\{u\geq t\})\, dt = \int_{(0,+\infty) \setminus Z} \P(\{u\geq t\})\, dt  \\ \leq&\;  \int_{(0,+\infty) \setminus Z} \P(\{v\geq t\})\, dt=\Fi(v), 
\end{aligned}\end{equation*}
hence $u$ is a minimizer.
\\
(ii) According to Proposition \ref{exmin1}, $u\in \W_0^{s,1}$ is also a weak solution. We use here an idea from \cite{novo}. Recalling that $\chi_{\{ u\geq t\} } \in \W_0^{s,1}(\Omega)$ thanks to \cite[Lemma 2.8]{bdlv}  and to \eqref{yu}, and picking any $F\subset \Omega$ of finite fractional perimeter, hence $\chi_F \in \W_0^{s,1}(\Omega)$, we can use them as test functions in the definition of weak solution, i.e. it holds that 
\begin{equation*} \begin{aligned} 
&\int_{\R^{2n} } \frac{\z(x,y)(\chi_F(x)-\chi_F(y))}{|x-y|^{n+s}}
\, dx dy - \int_{\R^{2n} } \frac{\z(x,y)(\chi_{\{ u\geq t\} } (x)-\chi_{\{ u\geq t\} } (y))}{|x-y|^{n+s}} 
\, dx dy  \\ 
=&\;  
	\int_\Omega f(x)(\chi_F - \chi_{\{ u\geq t\} } ) (x) \, dx.\end{aligned}\end{equation*}
	Since
	\[ \z(x,y)(u(x)-u(y))= |u(x)-u(y)|,\]
	and
	\[ u(x)-u(y)=\int_0^{+\infty} \left(\chi_{\{ u\geq t\}} (x) -\chi_{\{ u\geq t\}} (x) \right) dt, \]
	together with 
	\[
	 |u(x)-u(y)|=\int_0^{+\infty} |\chi_{\{ u\geq t\}} (x) -\chi_{\{ u\geq t\}} (x) | dt,\]
	 then
	 \[\z(x,y) \left(\chi_{\{ u\geq t\}} (x) -\chi_{\{ u\geq t\}} (x) \right) = |\chi_{\{ u\geq t\}} (x) -\chi_{\{ u\geq t\}} (x) | \]
	 for almost all $t\in (0,+\infty)$. 
Also
\[ \z(x,y)(\chi_F(x)-\chi_F(y))\leq |\chi_F(x)-\chi_F(y)|,\]
and we obtain that
\begin{equation*} \begin{aligned}&\int_\Omega f(x) (\chi_F - \chi_{\{ u\geq t\} } ) (x) \, dx \\
\leq	&\; 
 \int_{\R^{2n} } \frac{|\chi_F(x)-\chi_F(y)|}{|x-y|^{n+s}}
\, dx dy 
- \int_{\R^{2n} } \frac{|\chi_{\{ u\geq t\} } (x)-\chi_{\{ u\geq t\} } (y)|}{|x-y|^{n+s}}  dxdy
,\end{aligned}\end{equation*}
hence
\[ \P( \{u\geq t\} ) \leq \P(F),\]
and the proof is concluded.
\end{proof}

 \section{Examples of non-existence and non-uniqueness   when $f=1$} \label{examples}
 
In this section we take a closer look at  the torsion problem, i.e. \eqref{mainpb} with $f=1$ and denote
\begin{equation}  \begin{aligned}\label{torsion}  \J(u, \Omega)=\frac{1}{2}[u]_{W^{s,1}(\Rn)} -\int_\Omega u\, dx.\end{aligned}\end{equation}
We give examples of non-uniqueness encompassing open questions from both Section \ref{four} and \ref{minper}. Precisely, when $h_s(\Omega) =1 $ and when $|\Omega|^{\frac{s}{n}} = 1/(2S_{n,s})$, we give examples of non-uniqueness of minimizers. When $h_s(\Omega)>1 $ and when   $|\Omega|^{\frac{s}{n}} > 1/(2S_{n,s})$ we provide examples of sets for which no global minimizers exist.. In both cases, the example is provided by considering a ball $B_R$ large enough, and relying on the isoperimetric inequality \eqref{isoeq} and on $s$-calibrable sets, that we define. A set $\Omega$ is said to be $s$-calibrable if it is the $s$-Cheegar set of itself. It is known, see \cite[Remark 5.2]{brch}, that the ball is such a set, i.e.
\[ h_s(B_R)= \frac{\Per_s(B_R)}{|B_R|}= \inf_{A\subset B_R} \frac{\Per_s(A)}{|A|}.\] 
Owing to the Faber-Krahn inequality (see \cite[Proposition 5.5.]{brch}), it holds that
\begin{equation}  \begin{aligned} \label{fbk} h_s(\Omega) \geq |\Omega|^{-\frac{s}{n}}|B_1|^{\frac{s}{n}} h_s(B_1) = |\Omega|^{-\frac{s}{n}}  \frac{1}{2S_{n,s}} ,
\end{aligned}\end{equation}
hence $ |\Omega|^{-\frac{s}{n}}  \frac{1}{2S_{n,s}}\geq 1$ implies $h_s(\Omega) \geq 1$.
\\We were not able to prove  -- or disprove -- if there exists $\Omega$ such that 
\[h_s(\Omega)\geq 1, \qquad \mbox{ and }  |\Omega|^{\frac{s}{n}} > 1/(2S_{n,s}),\]
i.e., if minimizers exist but the bound of Theorem \ref{exmin} does not hold.

In case of equality, we can give an example of a non-trivial minimizer, and at the same time, of  non-uniqueness of minimizers.
\begin{proposition} \label{zxcv2} 
Let $\Omega =B_R$ be such that 
\[ h_s(B_R)=1.\]
Then  any $u\in \{\lambda \chi_{B_R} \, | \, \lambda\geq 0\} $ is a non-negative minimizer of $\J(\cdot, B_R)$. 
\end{proposition}

\begin{proof}
If $h_s(B_R)=1$ and since $B_R$ is s-calibrable, 
\[ 1= \frac{\Per_s(B_R)}{|B_R|},\] and this happens if 
\begin{equation}  \begin{aligned}\label{rrr}  R= (2S_{n,s})^{-\frac{1}{s}} |B_1|^{-\frac{1}{n}} =h_s(B_1)^{\frac{1}{s}}.\end{aligned}\end{equation}
We remark that in this case
\[ |B_R|^{\frac{s}{n}}=\frac{1}{2S_{n,s}}.\] 
Let $u=\chi_{B_R}$ be the characteristic function of $B_R$. Then 
\[\J (\chi_{B_R},B_R) = \P(B_R) .\] Rescaling by $R$ and using \eqref{isoeq}, we have that
\begin{equation}  \begin{aligned} \label{cvb} \J(\chi_{B_R},B_R)  = &\; R^{n-s}\Per_s(B_1) -|B_1|R^n = R^{n-s}\frac{1}{2S_{n,s} }|B_1|^{\frac{n-s}{n}} -R^n |B_1| 
\\ 
=&\; R^{n-s}\frac{1}{2S_{n,s} }|B_1|^{\frac{n-s}{n}} \left(1 -2R^s S_{n,s} |B_1|^{\frac{s}{n}}\right)=0,
\end{aligned}\end{equation}
by hypothesis.
We conclude the proof recalling Remark \ref{inf}. 
\end{proof}

On the other hand, we have the following. 
\begin{proposition} \label{zxcv1} 
Let $\Omega \supseteq B_R$ with 
\[ h_s(B_R)>1,\]
then 
\[\inf_{u \in \W_0^{s,1}(\Omega)} \J(u,\Omega)=-\infty.\]
\end{proposition}
\begin{proof}
Notice that $B_R \subseteq \Omega$ with $R$ from \eqref{rrr} implies that
\[|\Omega|^{\frac{s}{n}} > |B_R|>\frac{1}{2S_{n,s}}.\]
The computations in \eqref{cvb} yield 
\[ \J(\chi_{B_R}, \Omega)= \J(\chi_{B_R},B_R)<0\] and we conclude by homogeneity, see Remark \ref{inf}.  
\end{proof}

We point out that in \cite{kawohl}, the author studies $(-\Delta)_p u=1$ on $B_R$, proving that if $R\leq n$ then \eqref{up0} holds, and if $R>n$ then \eqref{infty} is in place. Notice that $h(B_1)=n$ and compare with our radius in \eqref{rrr}. Thus our results are sharp in the fractional case, and are the counterpart of those in \cite{kawohl}.

\smallskip

We can draw from \cite{brch} some other very interesting aspects of  sets $E \subset \Omega$ that minimize  $\P$, precisely 
 a regularity result following from \cite{criscaputo}. One can prove that these minimizing sets are almost minimal for the perimeter (basically, a perturbation of the set in a small ball  produces produced a term which is controlled by radius of the ball to a certain power).  Such sets are proved to have $C^1$ boundary outside of a set of singular points $\Sigma\subset \Omega$ such that $\dim_{\mathcal H}\Sigma \leq n-2$. Precisely, the authors of \cite[Proposition 6.4]{brch} prove that if $\tilde E\subset \Omega$ is $s$-Cheegar set, then $\partial E \cap \Omega$ is $C^1$ outside of a set of Hausdorff dimension at most $n-2$.
The proof remains unchanged if one analyzes minimizers of $\P$  for $f=1$. 
 
Further on, if $x_0\in \partial E\subset \Omega$ is a smooth point (i.e., there exists an interior and exterior tangent ball to $\partial E$ at $x_0$) then an Euler-Lagrange equation holds at that point. To be more precise, we define
 \[ H_s[E](x)= \lim_{\delta \to 0} \int_{\Rn\setminus B_\delta(x)} \frac{|\chi_{\Co E}(y)- \chi_E(y)|}{|x-y|^{n+s}}\, dy \]
 the fractional mean curvature of $\partial E$ at $x\in \partial E$. In \cite[Theorem 6.7]{brch} it is proved that if $E\subset \Omega$ is a minimizer of $2\Per_s(E)- h_s(\Omega) |E| $, then for any any smooth point $x\in \partial E$, it holds that
 \[ H_s[E](x)=h_s(\Omega),\]
 (we remark that the signs are opposite to those in \cite{brch} since we consider the definition of the opposite sign for the mean curvature). 
In exactly the same way, it follows that if $E$ minimizes $\P$ then for at all smooth points on $\partial E$ it holds  point-wisely that
 \[ H_s[E](x) =1.\]  

\section{"Flatness" of weak solutions and of minimizers}\label{two}

In \cite{Ponce}, the authors prove (using Stampacchia's truncation  method) that a weak solution of the $1$-Laplacian equation with $L^n(\Omega)$ right hand side data and zero boundary data has a  gradient that vanishes on a set of positive Lebesgue measure (and that the same holds for a right hand side with small norm in the Marcinkiewicz space or for a $BV$ minimizer of the associated energy). 
 
In this section, we prove similar results for the nonlocal problem, following the lines of the proofs in \cite{Ponce}. We believe that the results are interesting and worth a few words, in particular to overcome the difficulties arising from the nonlocal character of the problem. 

We emphasize that when 
$ \|f\|_{L^{\frac{n}{s}}(\Omega)}<(2S_{n,s})^{-1},$
 the unique minimizer is the null function and the flatness result is obvious.  However, the ``flatness'' still holds, independently on the size of the norm of $f$. 
 
We reiterate also that, as a consequence of the results in this section, in Definition \ref{weaksol}, one cannot get rid of   $\textbf{z}(x,y)$,  since in general weak solutions of \eqref{mainpb} 
for which $\textbf{z}\in \{-1,1\}$  up to sets of measure zero (hence, when $\textbf{z}(x,y)$ is the classical sign function of $u(x)-u(y)$) do not exist. Precisely, we have the following result.


Notice also that the results in this section work in $X_0^{s,1}(\Omega)$, with $\Omega$ only bounded and open,  without any assumption on  $\partial \Omega$. 
\begin{theorem}\label{thmone} Let $f\in L^{\frac{n}{s}}(\Omega)$. Let $u\in \W_0^{s,1}(\Omega)$ be a weak solution of \eqref{mainpb}.  Then the set
\[ \{(x,y) \in Q(\Omega) \, | \,  u(x)=u(y)\} \]
is of positive Lebesgue measure. In other words, there exists no function $u\in \W^{s,1}_0(\Omega)$ such that $u(x)\neq u(y)$ for almost any $(x, y)\in	Q(\Omega), x\neq y$, and such that $u$ is a weak solution of \eqref{mainpb}. 
\end{theorem}

\begin{proof}[Proof]
We recall that $Q(\Omega)=\R^{2n}\setminus (\Co \Omega)^2$. 
Suppose that such a $u$ described in the theorem exists. Then
\begin{equation*} \begin{aligned}
	\frac12 \iint_{\R^{2n}} \frac{\z(x,y)(w(v)-w(y))}{|x-y|^{n+s}} dx \, dy = \int_{\Omega} fw, \qquad \mbox{ for all }\; \; w\in \W^{s,1}_0(\Omega).
\end{aligned}\end{equation*}
Since $u(x)\neq u(y)$ for almost any $(x, y) \in Q (\Omega)$, we have that almost everywhere  in $Q( \Omega)$  the multi-valued function $\z(x,y)$ is the classical sign function and
\begin{equation}   \z(x,y) =
\left\{
\begin{aligned}
& 1, & \mbox{ if } & u(x)>u(y),\\
						-&1, & \mbox{ if } & u(x)<u(y).\end{aligned} \right. \end{equation}
We define for $k>0$,
	\begin{equation}  G_k(t)= 
\left\{
\begin{aligned}
		\label{gkt} & t+k, && t<-k\\
				& 0, &-&k\leq t\leq k\\
				& t-k, && t>k.
	\end{aligned} \right. \end{equation}
	Let us denote 
	\begin{equation}  \begin{aligned}\label{omegak} \Omega_k^1 := \{ x\in \Omega \; \big| \; u(x) <-& k\}, \qquad  \Omega_k^2 := \{ x\in \Omega \; \big| \; u(x) >k\}, \\   &\Omega_k:=  \Omega_k^1\cup  \Omega_k^2. \end{aligned}\end{equation}
	Notice that $G_k(u) \in \W^{s,1}_0(\Omega)$, since $G_k(u)=0$ almost everywhere in $\Co \Omega_k = (\Omega \setminus \Omega_k) \cup \Co \Omega$, and 
	\begin{equation*} \begin{aligned} \label{gf}
			  \int_\Omega \int_\Omega  \frac{ |G_k(u(x))- G_k(u(y))|}{|x-y|^{n+s}} dx\, dy  
			   = &\; \int_{\Omega_k} \int_{\Omega_k}  \frac{ |G_k(u(x))- G_k(u(y))|}{|x-y|^{n+s}} dx\, dy 
			  \\&\;  + 
		2 \int_{\Omega_k} \int_{\Omega \setminus \Omega_k}  \frac{ |G_k(u(x))- G_k(u(y))|}{|x-y|^{n+s}} dx\, dy. \end{aligned}\end{equation*}
	Now,
	\begin{equation*} \begin{aligned}
	& \int_{\Omega_k} \int_{\Omega_k}  \frac{ |G_k(u(x))- G_k(u(y))|}{|x-y|^{n+s}} dx\, dy 
	\\
			    \le &\;  \int_{\Omega_k^1} \int_{\Omega_k^1}  \frac{ |u(x)- u(y)|}{|x-y|^{n+s}} dx\, dy +
\int_{\Omega_k^2} \int_{\Omega_k^2}  \frac{ |u(x)- u(y)|}{|x-y|^{n+s}} dx\, dy 		
		\\ 
		&\; +
		 2 \int_{\Omega_k^1} \left(\int_{\Omega_k^2}  \frac{ |u(x)- u(y) |+ 2k}{|x-y|^{n+s}} \, dy  \right) dx
		 \le 4 [u]_{W^{s,1}(\Omega)} \\ &\; + 2\int_{\Omega_k^1}\left( \int_{\Omega_k^2}  \frac{u(y)-u(x) }{|x-y|^{n+s}} \, dy  \right) dx 
		 \le  6 [u]_{W^{s,1}(\Omega)},
	\end{aligned}\end{equation*}
	noting that  $k<-u(x)$ on $\Omega_k^1$ and $k<u(y)$ for $y\in \Omega_k^2$. 
		On the other hand
		\begin{equation*} \begin{aligned} & \int_{\Omega_k} \left(\int_{\Omega \setminus \Omega_k}  \frac{ |G_k(u(x))- G_k(u(y))|}{|x-y|^{n+s}} \, dy  \right) dx \\ = &\;
		2 \int_{\Omega_k^1} \left(\int_{\Omega \setminus \Omega_k}  \frac{ -u(x)-k}{|x-y|^{n+s}}\, dy\right) dx 
		+ 2 \int_{\Omega_k^2} \left(\int_{\Omega \setminus \Omega_k}  \frac{ u(x)-k}{|x-y|^{n+s}} \, dy\right) dx 
		\\ \leq&\;  2 \int_{\Omega_k^1} \left(\int_{\Omega \setminus \Omega_k}  \frac{ -u(x)+u(y) }{|x-y|^{n+s}}\, dy\right) dx 
	+ 2 \int_{\Omega_k^2} \left(\int_{\Omega \setminus \Omega_k}  \frac{ u(x)-u(y)}{|x-y|^{n+s}} \, dy\right) dx\\ 
	 \leq &\; 4[u]_{W^{s,1}(\Omega)}
		,\end{aligned}\end{equation*}
		noting that for $y\in\Omega \setminus \Omega_k$ we have that $-k\leq u(y) \le k$. 
	Furthermore 
	\[ \|G_k(u)\|_{L^1(\Omega)} \leq \|u\|_{L^1(\Omega)} + 2 k |\Omega|.\] This concludes the proof that $G_k(u) \in \W^{s,1}_0(\Omega)$. \\
	  	What is more, 
	  	\begin{equation*} \begin{aligned} 
	  	\,  [G_k(u)]_{W^{s,1}(\R^{n})} =&\; 
	 \int_{ \Omega_k} \int_{\Omega_k}  \frac{| G_k(u(x))- G_k(u(y))|}{|x-y|^{n+s}}dx\, dy
	\\ &\; + 
	 2  \int_{ \Omega_k} \left(\int_{\Co \Omega_k}  \frac{| G_k(u(x))|}{|x-y|^{n+s}}\, dy  \right) dx.
	  	\end{aligned}\end{equation*}
By H{\"o}lder's inequality and  the fractional Sobolev inequality in Theorem \ref{sobb},
	\begin{equation}  \begin{aligned} \label{yel2}
	 \left|\int_\Omega f (x) G_k(u(x)) \, dx\right| 
		\leq &\; \|f\|_{L^{\frac{n}s}(\Omega_k)}  \|G_k(u)\|_{L^{\frac{n}{n-s}}(\Omega_k)} \\ \le&\;  S_{n,s} \|f\|_{L^{\frac{n}s}(\Omega_k)}  [G_k(u)]_{W^{s,1}(\Rn)}.
	\end{aligned}\end{equation}
		We also point out that, since $\z(x,y) \in \sgn(u(x)-u(y)) $
	\[ \label{mas1} \z(x,y) \left(G_k(u(x))- G_k(u(y))\right) =\left|G_k(u(x))- G_k(u(y)\right|\]
	almost everywhere on $\R^{2n} \setminus (\Co \Omega_k)^2$.
	Indeed, this is obvious if $(x,y) \in (\Omega_k^1 \times \Omega_k^1 ) \cup  (\Omega_k^2 \times \Omega_k^2 )$, while, for instance, if $x\in \Omega_k^1, y\in \Omega_k^2$, since $u(y) >k>-k>u(x)$ and $\z(x,y)=-1$, it holds that 
	\[ \z(x,y) \left( G_k(u(x)) - G_k(u(y))\right)  = u(y)-u(x)-2k = \left|G_k(u(x)) - G_k(u(y))\right|.\] 
	The other cases can be settled with  similar observations. \\
		This and using  $G_k(u)$ as a test function in \eqref{ee1}, gives that			
	\begin{equation*} \begin{aligned}
 \frac12\iint_{\R^{2n}}   \frac{ |G_k(u(x))- G_k(u(y))|}{|x-y|^{n+s}} dx\, dy  = &\; \frac12\iint_{\R^{2n}} \z(x,y)  \frac{ G_k(u(x))- G_k(u(y))}{|x-y|^{n+s}} dx\, dy \\ 
	= &\; \int_\Omega f G_k(u) \, dx,
	\end{aligned}\end{equation*}
	  	hence by the Sobolev inequality
	\begin{equation*} \begin{aligned}
	 \frac{1}{2}[G_k(u)]_{W^{s,1}(\Rn) }
	  \leq S_{n,s} \|f\|_{L^{\frac{n}s}(\Omega_k)}  [G_k(u)]_{W^{s,1}(\Rn) },
	\end{aligned}\end{equation*}
that is, recalling that by \eqref{equinorm} the $W^{s,1}(\Rn)$ seminorm is finite,
	\begin{equation}  \begin{aligned} \label{yel3}
	\left(1 -2S_{n,s}  \|f\|_{L^{\frac{n}s}(\Omega_k)} \right)  [G_k(u)]_{W^{s,1}(\Rn) } \leq 0.	
	\end{aligned}\end{equation}
	Denote $T:= \|u\|_{L^\infty(\Omega)} \in [0,+\infty]$, then we claim that
	\[ | \{ |u|=T \} | =0.\]
	Indeed, if $T=+\infty$ it is obvious since $u\in L^1(\Omega)$. Otherwise, on the set $\{ |u|=T \} $ we have that $u(x)= u(y)$, but this can hold only almost everywhere, according to our reasoning by contradiction. From this and the Lebesgue dominated convergence theorem, we deduce that that
	\[
		\lim_{k \nearrow T} \|f\|_{L^{\frac{n}s}(\Omega \cap \{ |u|>k\} )} = \|f\|_{L^{\frac{n}s}(\Omega \cap \{ |u|=T\} )} =0.
	\]
	Consequently, for all $\varepsilon>0$ there exists $\tilde k \in (0,T)$ such that
	\[ 2 S_{n,s} \|f\|_{L^{\frac{n}s}(\Omega \cap \{ |u|>k\} )} <\varepsilon,
	\]
	for all $k \geq \tilde k$. 
	In particular, from \eqref{yel3}, we obtain
	\[
		 \|G_{\tilde k}(u)\|_{L^{\frac{n}{n-s}}(\Omega_{\tilde k})} \leq S_{n,s} [G_{\tilde k}(u)]_{W^{s,1}(\Rn) }  \leq 0.
	\]
	Hence $ |u| \leq \tilde k $ almost everywhere in $\Omega $, hence $T\leq \tilde k$.  This is in contradiction with the choice of $\tilde k$ and concludes the proof.	
\end{proof}

What is more, we are able to prove that minimizers of $\mathcal F^1_s$ are bounded in $\Omega$, and  that they reach the value of the $L^\infty$ norm on a set with positive Lebesgue measure. More precisely, we have the following. 

\begin{theorem}\label{thmtwo} Let  $f\in L^{\frac{n}s}(\Omega)$. Let $u\in \W_0^{s,1}(\Omega)$ be a minimizer of $ \mathcal F_1^s$. Then $u\in L^\infty(\Rn)$ and the set of extremal points
	\[ \big\{ x\in \Rn \; \big| \; |u(x)|=\|u\|_{L^\infty(\Rn)} \big\}\]
	has positive Lebesgue measure. 
\end{theorem}

\begin{proof}[Proof of Theorem \ref{thmtwo}]
We define, for $k>0$,
	\begin{equation*} T_k(t)=
	\left\{\begin{aligned} - &k, & \mbox{ if } & t<-k\\
					& t, & \mbox{ if } & -k\leq t\leq k\\
					& k, & \mbox{ if } & t>k.
					\end{aligned}\right.\end{equation*}
We recall the definition of $G_k(t)$, and remark that
		\begin{equation}  \begin{aligned} \label{mas3}
		t= T_k(t)+G_k(t),  \qquad
		| t-\tau|= |T_k(t)-T_k(\tau)| + |G_k(t)-G_k(\tau)|.
	\end{aligned}\end{equation}
	
	We notice that $T_k(u)=0$ in $\Co \Omega$ and we claim that  $T_k(u) \in X^{s,1}_0(\Omega)$. Using the notations in \eqref{omegak},  when $ (x,y) \in \Omega_ k^i\times \Omega_k^i$ for $i\in \{1,2\}$, $T_k(u(x))-T_k(u(y)) =0$. Then 
	\begin{equation*} \begin{aligned}
	&\int_\Omega \int_\Omega \frac{|T_k(u(x))-T_k(u(y))|}{|x-y|^{n+s}}\, dx dy  
	\\ 
	=&\; 2\int_{\Omega_k^1}  \left(\int_{\Omega_k^2} \frac{2k}{|x-y|^{n+s}}\, dy\right)dx 
	+ 2\int_{\Omega_k^1} \left(\int_{\Omega \setminus \Omega_k}\frac{u(y)+k}{|x-y|^{n+s}}\, dy \right) dx
\\	
&\;	+ 2\int_{\Omega_k^2} \left(\int_{\Omega \setminus \Omega_k}\frac{-u(y)+k}{|x-y|^{n+s}}\, dy \right) dx 
	+\int_{\Omega \setminus \Omega_k} \int_{\Omega \setminus \Omega_k}\frac{|u(x)-u(y)|}{|x-y|^{n+s}}\, dy dx
	\\
	\le &\; 2\int_{\Omega_k^1} \left(\int_{\Omega_k^2}  \frac{u(y)-u(x)}{|x-y|^{n+s}}\, dy\right) dx 
	+ 2\int_{\Omega_k^1} \left(\int_{\Omega \setminus \Omega_k}\frac{u(y)-u(x)}{|x-y|^{n+s}}\, dy \right) dx
\\	
&\;	+ 2\int_{\Omega_k^2} \left(\int_{\Omega \setminus \Omega_k}\frac{u(x)-u(y)}{|x-y|^{n+s}}\, dy \right) dx 
	+\int_{\Omega \setminus \Omega_k} \int_{\Omega \setminus \Omega_k}\frac{|u(x)-u(y)|}{|x-y|^{n+s}}\, dy dx
	\\
	\leq&\; C[u]_{W^{s,1}(\Omega)},
	\end{aligned}\end{equation*}
	since $-u>k$ in  $ \Omega_k^1$ and $k<u$ in $ \Omega_k^2$.
	Therefore $T_k(u)$ is a competitor for $u$, and from the minimality of $u$,
	\[
		 \F_1^s(u) \leq   \F_1^s(T_k(u)),
	\]
	that is
	\begin{equation*} \begin{aligned}
		&\iint_{\R^{2n}} \frac{ |u(x)-u(y)|- |T_k(u(x))-T_k(u(y))|}{|x-y|^{n+s}} dx \, dy\\  \leq &\; \int_\Omega f(x)\left( u(x)-T_k(u(x))\right) dx.
	\end{aligned}\end{equation*}
	From \eqref{mas3} we obtain that
	\begin{equation*} \begin{aligned} 
	\iint_{Q(\Omega)} \frac{  |G_k(u(x))-G_k(u(y))|}{|x-y|^{n+s}} dx \, dy \leq \int_\Omega f(x) G_k(u(x)) \, dx.
	\end{aligned}\end{equation*}
	This implies \eqref{yel3}, and also that
	\[ \left(1 -2S_{n,s}  \|f\|_{L^{\frac{n}s}(\Omega_k)} \right) \|G_k(u)\|_{L^{\frac{n}{n-s}}(\Omega)}\leq 0.\]
	Denoting $T:= \|u\|_{L^\infty(\Rn)} \in [0,\infty]$, we have that
	\begin{equation*} \begin{aligned} \|G_k(u)\|^{\frac{n-s}{n}}_{L^{\frac{n}{n-s}}(\Omega) }= \int_{\Omega_k^1} |u(x)+k|^{\frac{n-s}n}\, dx + \int_{\Omega_k^2} |u(x)-k|^{\frac{n-s}n}\, dx>0,\end{aligned}\end{equation*}
	so for all $0<k<T$
	\begin{equation}  \begin{aligned} \label{yu1} \|f\|_{L^{\frac{n}s}(\Omega_k)}  \geq \frac{1}{2S_{n,s}}.\end{aligned}\end{equation}
	Since $f\in {L^{\frac{n}s}(\Omega)}$,  and  $u$ is finite almost everywhere, it should hold that
		$ \| f\|_{L^{\frac{n}s}(\Omega_k)} $ is small for $k$ large enough. 
Suppose by contradiction that $T=+\infty$. Since $\Omega_k=\{ x\in \Omega \, | \, |u(x)|>k\}$, and $|\{u=+\infty\}|=0$, by Lebesgue's dominated convergence theorem
\[\lim_{k\to +\infty} \| f\|_{L^{\frac{n}s}(\Omega_k)} = \| f\|_{L^{\frac{n}s}(\{u=+\infty\})}=0,\]
hence for any $\varepsilon>0$ there is some $\tilde k>0$ such that for all $k>\tilde k$ 	
\[ \| f\|_{L^{\frac{n}s}(\Omega_k)}<\varepsilon.\] This gives a contradiction to \eqref{yu1} and 
		implies that $T<\infty$. It follows that 
		$u\in L^\infty(\Rn)$. 
		Furthermore, letting $k \nearrow T$ in \eqref{yu1}, we observe that
		\[ \lim_{k \nearrow T } \|f\|_{L^{\frac{n}s}(\Omega_k)} = \|f\|_{L^{\frac{n}s}\left(\{ |u|= T\}\right ) } \geq \frac{1}{2S_{n,s}},
		\]
		hence $\{ |u|= \|u\|_{L^\infty(\Rn)} \}$ has positive Lebesgue measure, as stated. 	
\end{proof}

\section{Appendix A. Pointwise limit}

In this section, we are interested in the behavior of the limit as $p\to 1$ of the kinetic part of the energy, specifically what we have denote in \eqref{ken} by $\mathcal E_p^s$. What we want to emphasize is that the fractional parameter $s_p$, used throughout this paper, arises naturally when looking at this pointwise limit. To broaden the scope of our statement, we allow for non-vanishing exterior data under specific additional conditions.  These prerequisites are expressed in terms of the so-called nonlocal tail of $u$, see \cite{bdlv}, precisely for $x\in \Omega$
\[\tail_{s_p}^p(u,x) =\int_{\Co \Omega} \frac{|u(y)|^p}{|x-y|^{(n+s)p}}\, dy,\]
basically encompassing the contribution of the exterior data to the energy. This concept is an adaptation for our framework of the  nonlocal tail first introduced in papers such as \cite{MR3542614,MR3237774}.

\begin{theorem}\label{thlim}
Let	$q\in (1,c_{n,s})$, where $c_{n,s}$ is such that $s_q\in (s,1)$ and $ s_q q<1$.  
Let  $u\colon \Rn \to \R$ be such that $u\in \W^{s_q,q}(\Omega) $ and
\begin{equation}  \begin{aligned} \label{fru} 
	\tail_{s}^1(u,\cdot),\, \, \tail_{s_q}^q(u,\cdot)\in L^1(\Omega).\end{aligned}\end{equation}
	Then it holds that
\[ \lim_{p\to 1} \E^{s_p}_p(u,\Omega)= \E^{s}_1(u,\Omega).\]

\end{theorem}

\begin{proof}

Let $s_0\in(s,1), \delta  \in(0 ,1-s)$. Since $p$ converges to $1$, it is safe to assume that 
\[ p \leq \min\left\{q, \frac{n-\delta}{n+1-s} , \frac{n+s_0}{n+s} \right\} .\]
 Notice that
\[ (n+s)q=n+s_qq.\]
We remark that if $u\in \W^{s_q,q}(\Omega)$, then by the H{\"o}lder inequality we have that $u\in \W^{s_t,t}(\Omega)$ for all $t\in [1, q)$. Indeed
\begin{equation}  \begin{aligned}\label{of3}  &\int_{\Omega}\int_\Omega \frac{|u(x)-u(y)|^t}{|x-y|^{n+s_t t} } \, dx dy
\leq [ u]_{W^{s_q,q}(\Omega)}^t |\Omega|^{\frac{2(q-t)}{q}},
\quad \mbox{ and }\\ 
 &\;\int_\Omega |u|^t \, dx \leq \|u\|_{L^q(\Omega)}^t |\Omega|^{\frac{q-t}q}.\end{aligned}\end{equation}
Furthermore, \eqref{fru} implies that $\tail_{s_t}^t (u,\cdot) \in L^1(\Omega)$ for all $t\in (1,q)$. Indeed, there exists a unique $\tau\in (0,1)$ such that $t=\tau+(1-\tau)q$, and by Young inequality we have that for $a,b\geq 0$,
\[a^\tau b^{(1-\tau)q} \leq \tau a + (1-\tau) b^q \leq a + b^q ,\]
which,  for $a=b$ becomes
\begin{equation}  \begin{aligned}\label{youngorig} a^t\leq a + a^q \end{aligned}\end{equation}
for all $t\in (1,q)$.
This applied to $a= \displaystyle{|g(y)|}/{|x-y|^{n+s}}$,   gives
\begin{equation}  \begin{aligned}\label{young}
\frac{ |u(y)|^t}{|x-y|^{(n+s)t}} \leq \frac{|u(y)|}{|x-y|^{n+s}} + \frac{|u(y)|^q}{|x-y|^{(n+s)q}},	
\end{aligned}\end{equation}
hence integrating, we get that $\tail_{s_t}^t (u,\cdot) \in L^1(\Omega)$ for all $t\in (1,q)$, as desired.

We prove the thesis of the theorem  in two steps, first for smooth, compactly supported functions, and then conclude by density.
\medskip

\noindent \emph{Step 1.} We prove that  
\begin{equation}  \begin{aligned} \label{of2}\lim_{p\to 1} \E^{s_p}_p(\psi,\Omega)= \E^{s}_1(\psi,\Omega), \end{aligned}\end{equation}
for all $\psi \colon \Rn \to \R$, such that $\psi \in C^\infty_0(\Omega)$ and $\psi=u$ on $\Co \Omega$. \\
To estimate the contribution in $\Omega \times \Omega$, we notice  that if $|x-y|\geq 1$, 
\[ \frac{1}{p} \frac{|\psi(x)-\psi(y)|^p}{|x-y|^{n+s_pp}} \leq \frac{2\|\psi\|^p_{L^\infty(\Omega)}}{|x-y|^{n+s_pp}} 
\leq \frac{C_1}{|x-y|^{n+s}} ,\]
with $C_1=C_1(\|\psi\|_{L^\infty(\Omega)})>0$. On the other hand, when
$|x-y|<1$,  
\[\frac{1}{p} \frac{|\psi(x)-\psi(y)|^p}{|x-y|^{n+s_pp}}  
	\leq \frac{\|\nabla \psi\|_{L^\infty(\Omega)}^p |x-y|^p}{|x-y|^{n+s_pp}} \leq \frac{C_2}
	{|x-y|^{n-\delta}} ,
\]
with $C_2=C_2(\|\nabla \psi\|_{L^\infty(\Omega)})>0$,
recalling that $(n+s-1)p<n-\delta$ by choice of $\delta$. Now,
\[ \iint_{\left(\Omega \times \Omega\right) \cap \{|x-y|\geq 1\} } \frac{dx \, dy}{|x-y|^{n+s}} \leq \iint_{\Omega \times \Omega } dx \, dy =|\Omega|^2 ,\]
and
\[ \iint_{\left(\Omega \times \Omega\right) \cap \{ |x-y|< 1\}} \frac{dx \, dy}{|x-y|^{n-\delta}} \leq \int_\Omega \left(\int_{B_1(x)}  \frac{ dy}{|x-y|^{n-\delta}}\right) dx \leq C(\Omega,\delta,n). \]
By the dominated convergence theorem we have that
\begin{equation}  \begin{aligned}\label{of1} \lim_{p \to 1} \frac{1}{p}\iint_{\Omega \times \Omega}  \frac{|\psi(x)-\psi(y)|^p}{|x-y|^{n+s_pp}}  \,dx dy= \iint_{\Omega \times \Omega} \frac{|\psi(x)-\psi(y)|}{|x-y|^{n+s}} \,dx dy .
\end{aligned}\end{equation}
Furthermore, for the nonlocal contribution $(x,y)\in \Omega \times \Co \Omega$, 
\begin{equation*} \begin{aligned}\frac{1}{p} \frac{|\psi(x)-u(y)|^p}{|x-y|^{n+s_pp}} \leq 
\frac{2^{p-1}(|\psi(x)|^p + |u(y)|^p)}{|x-y|^{n+s_pp}}
\leq \frac{C_3}{|x-y|^{n+s_pp}}   +\frac{2|u(y)|^p}{|x-y|^{n+s_pp}}  ,
\end{aligned}\end{equation*}
with $C_3=C_1(\|\psi\|_{L^\infty(\Omega)})>0$. 
When $|x-y|\geq 1$ then $|x-y|^{n+s_pp} \geq |x-y|^{n+s} $
and when $|x-y|\leq 1$ then $|x-y|^{n+s_pp} \geq |x-y|^{n+s_0} $ by the choice of $s_0$. 
Together with \eqref{young}, we have then
\begin{equation*} \begin{aligned}\frac{1}{p} \frac{|\psi(x)-u(y)|^p}{|x-y|^{n+s_pp}} \leq  \frac{C_3}{|x-y|^{n+\sigma}} +\frac{2|u(y)|}{|x-y|^{n+s}} + \frac{2|u(y)|^q}{|x-y|^{(n+s)q}}  
\end{aligned}\end{equation*}
with either $\sigma=s$ or $\sigma=s_0$, for the suitable cases $|x-y|$ less or greater than $1$. Recalling \eqref{fru} and  that
\[\Per_\sigma(\Omega,\Rn)= \iint_{\Omega \times \Co \Omega } \frac{dx \, dy}{|x-y|^{n+\sigma}}<+\infty ,\]
the right hand side is an $L^1(\Omega \times \Co \Omega)$ and the dominated convergence theorem gives that
\[\lim_{p\to 1} \frac{1}p \iint_{\Omega \times \Co \Omega } \frac{|\psi(x)-u(y)|^p}{|x-y|^{n+s_pp}}\, dx dy = \iint_{\Omega \times \Co \Omega } \frac{|\psi(x)-u(y)|}{|x-y|^{n+s}} \, dx dy. \]
This, together with \eqref{of1}, concludes \eqref{of2}.

\medskip

\emph{Step 2.} By the density of $C^\infty_c(\Omega)$ in $\W^{s_q,q}(\Omega)$, we have that for $u\in \W^{s_q,q}(\Omega)$, there exists $\psi_j \colon \Omega \to \R$ such that $\psi_j \in C^\infty_c(\Omega)$  and
\[ \| \psi_j-u\|_{W^{s_q,q}(\Omega)} \longrightarrow 0, \quad \mbox{ as } j \to \infty.\]
Without changing notations, we take  $\psi_j\colon \Rn \to \R$ such that  $\psi_j=u$ on $\Co \Omega$. 
From \eqref{of3} it holds
\[  \| \psi_j-u\|_{W^{s_t,t}(\Omega)} \longrightarrow 0, \quad \mbox{ as } j \to \infty\]
 for all $t\in [1,q]$. 
We  have
\begin{equation*} \begin{aligned}\lim_{p\to 1} \left(\E^{s_p}_p(u,\Omega)- \E^{s}_1(u,\Omega)\right) = &\; \lim_{j \to +\infty} \lim_{p\to 1} \left(\E^{s_p}_p(u,\Omega)- \E^{s_p}_p(\psi_j,\Omega)\right)
	\\ &\;+\lim_{j \to +\infty} \lim_{p\to 1} \left(\E^{s_p}_p(\psi_j,\Omega)- \E^{s}_1(\psi_j,\Omega)\right) \\
	&\; +  \lim_{j \to +\infty} \left(\E^{s}_1(\psi_j,\Omega)- \E^{s}_1(u,\Omega)\right)
\\
= &\; L_1+L_2+L_3.
\end{aligned}\end{equation*}
Notice that $L_2=0$ by Step 1, and $L_3=0$ since
\begin{equation*} \begin{aligned} \left| \E^{s}_1(\psi_j,\Omega)- \E^{s}_11(u,\Omega)\right| \leq&\; \iint_{Q(\Omega) } \frac{|(\psi_j-u)(x)- (\psi_j-u)(y)|}{|x-y|^{n+s}}\, dy dx
\\
\leq&\;  C_{n,s,\Omega} \| \psi_j -u \|_{W^{s,1}(\Omega)} ,\end{aligned}\end{equation*}
using also \eqref{frach}. To prove that also $L_1=0$, we proceed as follows. We use the inequalities $||a|^p-|b|^p |\leq p(|a|^{p-1}+ |b|^{p-1}) |a-b|$ and H{\"o}lder's  to  get that
\begin{equation*} \begin{aligned}
& \left|\E^{s_p}_p(u,\Omega)-\E^{s_p}_p(\psi_j,\Omega)\right| 
\\
\leq &\; 2\left(\iint_{Q(\Omega) } \frac{ |(u-\psi_j)(x)-(u-\psi_j)(y)|^p}{|x-y|^{(n+s)p}}\, dx dy\right)^{\frac{1}p} 
\\  &\; \left[ \left(\iint_{Q(\Omega) } \frac{ |u(x)-u(y)|^{p} }{|x-y|^{(n+s)p}}\right)^{\frac{p-1}{p}}+\left(\iint_{Q(\Omega) } \frac{ |\psi_j(x)-\psi_j(y)|^{p})}{|x-y|^{(n+s)p}}\right)^{\frac{p-1}{p}}\right]
\\
:=&\; 2 I(j,p) \left[  J(p) +  K(j,p)\right].
\end{aligned}\end{equation*}
Now, using \eqref{youngorig} for 
$a= |(u-\psi_j)(x)-(u-\psi_j)(y)|/{|x-y|^{(n+s)}}$ 
\begin{equation*} \begin{aligned} I(j,p)^p \leq &\; \iint_{Q(\Omega)}  \frac{|(u-\psi_j)(x)-(u-\psi_j)(y)|}{|x-y|^{n+s}} \, dx dy \\ &\; +
	 \iint_{Q(\Omega)}  \frac{|(u-\psi_j)(x)-(u-\psi_j)(y)|^q}{|x-y|^{(n+s)q}} \, dx dy 
	 \\
	 \leq &\; 
	[u-\psi_j]_{W^{s,1}(\Omega)} +	[u-\psi_j]^q_{W^{s_q,q}(\Omega)} 
	\\
	&\; +2 \iint_{\Omega \times \Co \Omega } \frac{|(u-\psi_j)(x)|   }{|x-y|^{n+s}}\, dx dy
	+ 2 \iint_{\Omega \times \Co \Omega } \frac{|(u-\psi_j)(x)|^q   }{|x-y|^{n+s_qq}}\, dx dy
	\\
	\leq &\; C_{n,s,q,\Omega} \left(   \|u-\psi_j\|_{W^{s,1}(\Omega)} + \|u-\psi_j\|^q_{W^{s_q,q}(\Omega)}\right)  
	 ,
\end{aligned}\end{equation*}
by \eqref{frach}  and recalling that $u=\psi_j$ on $\Co \Omega$. 
Renaming the constants and using \eqref{of3}, we get
that
\[  I(j,p) \leq C_{n,s,q,\Omega} \|u-\psi_j\|_{W^{s_q,q}(\Omega)}, \]
from which follows that
\[\lim_{j \to +\infty} \lim_{p\to 1}  I(j,p) =0. \]
On the other hand, again by \eqref{youngorig} and \eqref{frach},
\begin{equation*} \begin{aligned} &\; J(p)^{\frac{p}{p-1}}  \leq 
	C_{n,s,q,\Omega} \left(\|u\|_{W^{s,1}(\Omega)} + \|u\|_{W^{s_q,q}(\Omega)}^q \right.
	\\
	&\;\left. +  \iint_{\Omega \times \Co \Omega } \frac{|u(y)|}{|x-y|^{n+s}} dy \, dx
	+ \iint_{\Omega \times \Co \Omega } \frac{|u(y)|^q}{|x-y|^{(n+s)q}} dy \, dx
	\right) 
	\\
	\leq &\; 		C_{n,s,q,\Omega} \left(\|u\|_{W^{s,1}(\Omega)} + \|u\|_{W^{s_q,q}(\Omega)}^q   \|\tail_{s}(u,\cdot )\|_{L^1(\Omega)} +  \|\tail^q_{s_q}(u,\cdot )\|_{L^1(\Omega)} \right)
\end{aligned}\end{equation*}
using \eqref{of3} and renaming the constants.
Thus $J(p)$ can be bounded from above, uniformly in $p$.
Finally, in the same way, 
\begin{equation*} \begin{aligned} K(j,p)^{\frac{p}{p-1}}  \leq &\; 
	C_{n,s,q,\Omega} \left( \|\psi_j\|_{W^{s,1}(\Omega)}+ \|\psi_j\|_{W^{s_q,q}(\Omega)}^q + \|\tail_{s}(u,\Co \Omega )\|_{L^1(\Omega)} \right.\\
	&\; \left. + \|\tail^q_{s_q}(u,\cdot )\|_{L^1(\Omega)} \right)
,\end{aligned}\end{equation*}
using that for $j$ large enough,
\[ \|\psi_j\|_{W^{s_t,t}(\Omega)}\leq  \|u\|_{W^{s_t,t}(\Omega)} +\|u-\psi_j\|_{W^{s_t,t}(\Omega)} \leq \|u\|_{W^{s_t,t}(\Omega)} +1\] for all $t\in [1,q]$. Thus
also $K(j,p)$ can be bounded from above uniformly in $p$.
It follows that
\[\lim_{j \to +\infty} \lim_{p\to 1}  I(j,p)(J(p)+K(j,p)) =0, \]
hence that $L_1=0$. This concludes the proof of the theorem. 
\end{proof}

We remind the reader that according to \cite[Lemma 2.3]{bdlvm}, \eqref{fru} can be achieved with $u\in W^{s,q}(\Co \Omega)$ -- or can be sharpened, by requiring a combination of conditions near the boundary and far from the boundary of $\Omega$. Notice also that when $\varphi=0$, such a requirement is satisfied, and we can write the following corollary.

\begin{corollary}
Let	$q\in (1,c_{n,s})$, where $c_{n,s}$ is such that $s_q\in (s,1)$ and $ s_q q<1$. 
Let  $u\colon \Rn \to \R$ be such that $u\in \W^{s_q,q}_0(\Omega)$. Then
\[ \lim_{p\to 1} \F^{s_p}_p(u)= \F^{s}_1(u).\]
\end{corollary}

\section{Appendix B. The $(s,p)$-problem} 

For completeness of the exposition and for the reader's benefit, we prove in this appendix some basic facts about minimizers/weak solutions of the $(s,p)$-problem, that we have used in this note. 

Notice also that his result holds in $X_0^{s,p}(\Omega)$ with $\Omega$ bounded and open. 
\begin{theorem}\label{pexist} Let $p>1$, $\gamma \geq \frac{n}{s_p}$ and let $f \in L^\gamma(\Omega)$.  There exists $u_p\in 	\W_0^{s,p}(\Omega)$, the unique  $(s,p)$-minimizer and weak solution of \eqref{mainpbp}. 
\label{exp}
\end{theorem}

\begin{proof}
Recalling that \[ p_s^* =\frac{np}{n-sp},\]
notice that 
\[ \frac{n}{s_p}=\frac{np}{np-n+sp} = (p_s^*)',\]
the conjugate Sobolev exponent.  
We point out that the constants may change value from line to line, indicating a positive quantity, possibly depending on $s,p,n, \gamma,\Omega$. Using the H{\"o}lder and Sobolev inequality, we get that
\begin{equation}  \begin{aligned} \label{holder}\left|\int_\Omega fu \, dx\right| \leq \|f\|_{L^{\frac{n}{s_p}}(\Omega)} \|u\|_{L^{p_s^*}(\Omega)} \leq S^{\frac{1}{p}}_{n,s,p}\|f\|_{L^{\gamma}(\Omega)} [u]_{W^{s,p}(\R^n)}
.
\end{aligned}\end{equation}
Using the Young inequality, we further have for some fixed $\varepsilon\in (0,1/2)$,
\begin{equation*} \begin{aligned} \|f\|_{L^{\gamma}(\Omega)}[u]_{W^{s,p}(\Rn)} \leq &\;  \varepsilon \frac{[u]_{W^{s,p}(\Rn)}^p}{p} + \frac{p-1}{p\varepsilon^\frac{1}{p-1}} \left(\|f\|_{L^{\gamma}(\Omega)}\right)^{\frac{p}{p-1}} 
=&\; :\varepsilon \frac{[u]_{W^{s,p}(\Rn)}^p}{p}  + C_\varepsilon \|f\|_{L^{\gamma}(\Omega)}^{\frac{p}{p-1}}
.\end{aligned}\end{equation*}
It follows that
\begin{equation*} \begin{aligned} \F_p^s(u,\Omega) \geq\frac{1}{2p}[ u]^p_{W^{s,p}(\Rn)}  - \int_\Omega fu \, dx 
 =&\;  \frac{1}{p} \left(\frac{1}{2}-\varepsilon\right) [u]^p_{W^{s,p}(\Rn)} -  C_{\varepsilon} \|f\|_{L^{\gamma}(\Omega)}^{\frac{p}{p-1}} 
 \geq - C_{\varepsilon} \|f\|_{L^{\gamma}(\Omega)}^{\frac{p}{p-1}}.\end{aligned}\end{equation*}
Thus the energy is bounded from below, and it follows that there exists a minimizing sequence. Let 
$\{u_k\}_k \in\ \W_0^{s,p}(\Omega)$ be such that
\[ \lim_{k\to \infty} \F_p^s (u_k,\Omega)= \inf_{u\in \W^{s,}(\Omega)} \F_p^s(u,\Omega)=:m \geq- C_\varepsilon \|f\|_{L^{\gamma}(\Omega)}^{\frac{p}{p-1}} .\]
There exists $\overline k\in \N$ such that for all $k>\overline k$
\[\F_p^s(u_k) <m+1.\] 
By this and by \eqref{glows1}, we obtain that for all $B_R \supset \supset \Omega$,
\[  \|u_k\|^p_{W^{s,p}(B_R)} \leq C_2 [u_k]^p_{W^{s,p}(\Rn)}  \leq C\left( m+1+  C_\varepsilon \|f\|_{L^{\gamma}(\Omega)}^{\frac{p}{p-1}} \right) .\] 
By compactness, and the fact that $u_k =0$ in $\Co \Omega$, there exists $u\in \W^{s,p}_0(\Omega)$ and  a subsequence 
\[ u_{k_i }\to u \qquad \mbox{ in } L^p(\Omega) \, \mbox{ and a.e. in }\,  \,   \Rn.\] 
Further, $\|u_{k_i}\|_{L^{\frac{np}{n-sp}}(\Omega)}$ is uniformly bounded by the Sobolev inequality, hence up to taking a subsequence of $ u_{k_i}$ that we still call $u_{k_i}$, 
\[ u_{k_i} \longrightarrow u \]
weakly in $L^{\frac{np}{n-sp}}(\Omega)$. By Fatou it holds that
\[ \iint_{\R^{2n}}  \frac{|u(x)-u(y)|^p}{|x-y|^{n+sp}} dx dy\leq \liminf_i \iint_{\R^{2n}}  \frac{|u_{k_i}(x)-u_{k_i}(y)|^p}{|x-y|^{n+sp}} dx dy  ,\]
 and this summed up with the weak convergence in $L^{\frac{np}{n-sp}}(\Omega)$, gives that
\[ \F_p^s(u) \leq m,\]
hence $u$ is a minimizer. \\
To see that the minimizer $u$ is also a weak solution, i.e. that $u$ solves the corresponding Euler-Lagrange equation, one takes in a standard way a perturbation of $u$ with a test function $\varphi \in \W_0^{s,p}(\Omega)$ and deduces $u$ is a weak solution by using that the first variation of the energy vanishes, 
\[ \frac{d}{dt}\F_p^s(u+t\varphi) \big|_{t=0}=0.\] As customary, if $u$ is a weak solution and $v$ is any competitor for $u$, we consider as test function $w=u-v$, and the fact that $u$ is a minimizer is obtained by using the Young inequality. \\
Finally,  uniqueness follows by strict convexity of the energy $\F_p^s$.
\end{proof}

Since we need the existence of a $(s_p,p)$-minimizer, we clarify the following corollary.
\begin{corollary} Let $p>1$ and let  $f \in L^{\frac{n}{s}}(\Omega)$.  There exists $u_p\in \W_0^{s_p,p}(\Omega)$, the unique  $(s_p,p)$-minimizer and weak solution of \eqref{mainpbp}. 
\label{expcor}
\end{corollary}

\begin{proof}
The proof is immediate, noticing that 
\[ \gamma=\frac{n}{s}\geq \frac{n}{(s_p)_p}=\frac{n}{n+s_p- \frac{n}{p}} =\frac{np}{(2n+s)p-2n} \] 
and applying Theorem \ref{pexist}. 
\end{proof}

\bibliography{biblio}

\begin{thebibliography}{10}

\bibitem{BDG}
Enrico Bombieri, Ennio De~Giorgi, and Enrico Giusti.
\newblock Minimal cones and the {B}ernstein problem.
\newblock {\em Invent. Math.}, 7:243--268, 1969.

\bibitem{brch}
Lorenzo Brasco, Erik Lindgren, and Enea Parini.
\newblock The fractional {C}heeger problem.
\newblock {\em Interfaces Free Bound.}, 16(3):419--458, 2014.

\bibitem{bws1}
Claudia Bucur.
\newblock Some notes on functions of least {$W^{s,1}$}-fractional seminorm.
\newblock {\em Bruno Pini Mathematical Analysis Seminar.}
\newblock 14:38--57, 2023.

\bibitem{bdlvm}
Claudia Bucur, Serena Dipierro, Luca Lombardini, Jos\'{e}~M. Maz\'{o}n, and
  Enrico Valdinoci.
\newblock {$(s, p)$}-harmonic approximation of functions of least
  {$W^{s,1}$}-seminorm.
\newblock {\em Int. Math. Res. Not. IMRN}, (2):1173--1235, 2023.

\bibitem{bdlv}
Claudia Bucur, Serena Dipierro, Luca Lombardini, and Enrico Valdinoci.
\newblock Minimisers of a fractional seminorm and nonlocal minimal surfaces.
\newblock {\em Interfaces Free Bound.}, 22(4):465--504, 2020.

\bibitem{bueno2}
Hamilton Bueno and Grey Ercole.
\newblock Solutions of the {C}heeger problem via torsion functions.
\newblock {\em J. Math. Anal. Appl.}, 381(1):263--279, 2011.

\bibitem{bueno}
Hamilton Bueno, Grey Ercole, and Shirley~S. Macedo.
\newblock Asymptotic behavior of the {$p$}-torsion functions as {$p$} goes to
  1.
\newblock {\em Arch. Math. (Basel)}, 107(1):63--72, 2016.

\bibitem{cheeg}
Hamilton~P. Bueno, Grey Ercole, Shirley~S. Macedo, and Gilberto~A. Pereira.
\newblock Torsion functions and the {C}heeger problem: a fractional approach.
\newblock {\em Adv. Nonlinear Stud.}, 16(4):689--697, 2016.

\bibitem{nms}
Luis Caffarelli, Jean-Michel Roquejoffre, and Ovidiu Savin.
\newblock Nonlocal minimal surfaces.
\newblock {\em Comm. Pure Appl. Math.}, 63(9):1111--1144, 2010.

\bibitem{criscaputo}
M.~Cristina Caputo and Nestor Guillen.
\newblock Regularity for non-local almost minimal boundaries and applications.
\newblock {\em arXiv preprint arXiv:1003.2470}, 2011.

\bibitem{trombetti}
Marco Cicalese and Cristina Trombetti.
\newblock Asymptotic behaviour of solutions to {$p$}-{L}aplacian equation.
\newblock {\em Asymptot. Anal.}, 35(1):27--40, 2003.

\bibitem{teolu}
Matteo Cozzi and Luca Lombardini.
\newblock On nonlocal minimal graphs.
\newblock {\em Calc. Var. Partial Differential Equations}, 60(4):Paper No. 136,
  72, 2021.

\bibitem{MR3237774}
Agnese Di~Castro, Tuomo Kuusi, and Giampiero Palatucci.
\newblock Nonlocal {H}arnack inequalities.
\newblock {\em J. Funct. Anal.}, 267(6):1807--1836, 2014.

\bibitem{MR3542614}
Agnese Di~Castro, Tuomo Kuusi, and Giampiero Palatucci.
\newblock Local behavior of fractional {$p$}-minimizers.
\newblock {\em Ann. Inst. H. Poincar\'{e} Anal. Non Lin\'{e}aire},
  33(5):1279--1299, 2016.

\bibitem{hitch}
Eleonora Di~Nezza, Giampiero Palatucci, and Enrico Valdinoci.
\newblock Hitchhiker's guide to the fractional {S}obolev spaces.
\newblock {\em Bull. Sci. Math.}, 136(5):521--573, 2012.

\bibitem{fisc}
Alessio Fiscella, Raffaella Servadei, and Enrico Valdinoci.
\newblock Density properties for fractional {S}obolev spaces.
\newblock {\em Ann. Acad. Sci. Fenn. Math.}, 40(1):235--253, 2015.

\bibitem{Frank}
Rupert~L. Frank and Robert Seiringer.
\newblock Non-linear ground state representations and sharp {H}ardy
  inequalities.
\newblock {\em J. Funct. Anal.}, 255(12):3407--3430, 2008.

\bibitem{morini}
Nicola Fusco, Vincent Millot, and Massimiliano Morini.
\newblock A quantitative isoperimetric inequality for fractional perimeters.
\newblock {\em J. Funct. Anal.}, 261(3):697--715, 2011.

\bibitem{grisvard}
Pierre Grisvard.
\newblock {\em Elliptic problems in nonsmooth domains}, volume~69 of {\em
  Classics in Applied Mathematics}.
\newblock Society for Industrial and Applied Mathematics (SIAM), Philadelphia,
  PA, 2011.
\newblock Reprint of the 1985 original [MR0775683], With a foreword by Susanne
  C. Brenner.

\bibitem{J}
Petri Juutinen.
\newblock {$p$}-{H}armonic approximation of functions of least gradient.
\newblock {\em Indiana Univ. Math. J.}, 54(4):1015--1029, 2005.

\bibitem{kawohl}
Bernhard Kawohl.
\newblock On a family of torsional creep problems.
\newblock {\em J. Reine Angew. Math.}, 410:1--22, 1990.

\bibitem{another}
Dingding Li and Chao Zhang.
\newblock On the solutions of nonlocal 1-laplacian equation with $ l^1 $-data.
\newblock {\em Discrete and Continuous Dynamical Systems}, 44(5):1354--1375,
  2024.

\bibitem{robin}
Jos\'{e}~M. Maz\'{o}n, Julio~D. Rossi, and Sergio Segura~de Le\'{o}n.
\newblock The 1-{L}aplacian elliptic equation with inhomogeneous {R}obin
  boundary conditions.
\newblock {\em Differential Integral Equations}, 28(5-6):409--430, 2015.

\bibitem{mazfr}
Jos\'{e}~M. Maz\'{o}n, Julio~D. Rossi, and Juli\'{a}n Toledo.
\newblock Fractional {$p$}-{L}aplacian evolution equations.
\newblock {\em J. Math. Pures Appl. (9)}, 105(6):810--844, 2016.

\bibitem{novo}
Matteo Novaga and Fumihiko Onoue.
\newblock Local {H}\"{o}lder regularity of minimizers for nonlocal denoising
  problems.
\newblock {\em arXiv preprint arXiv:2107.08106}, 2021.

\bibitem{Ponce}
Luigi Orsina and Augusto~C. Ponce.
\newblock Flat solutions of the 1-{L}aplacian equation.
\newblock {\em Rend. Istit. Mat. Univ. Trieste}, 49:41--51, 2017.

\bibitem{SWZ}
Peter Sternberg, Graham Williams, and William~P. Ziemer.
\newblock Existence, uniqueness, and regularity for functions of least
  gradient.
\newblock {\em J. Reine Angew. Math.}, 430:35--60, 1992.

\bibitem{visintin}
Augusto Visintin.
\newblock Generalized coarea formula and fractal sets.
\newblock {\em Japan J. Indust. Appl. Math.}, 8(2):175--201, 1991.

\end{thebibliography}
\bibliographystyle{plain}

\end{document}